\documentclass[a4paper, 11pt, fleqn]{amsart}

\usepackage{amsmath, amssymb, graphicx, subfig, amsrefs}
\usepackage{listings}

\def\Stars{\mathrm{Stars}}
\def\aaa{\mathbf{a}}
\def\bb{\mathbf{b}}
\def\pp{\mathbf{p}}
\def\qq{\mathbf{q}}
\def\ss{\mathbf{s}}
\def\vv{\mathbf{v}}
\def\xx{\mathbf{x}}
\def\nn{\mathbf{n}}
\def\ee{\mathbf{e}}
\def\ff{\mathbf{f}}

\def\ii{\mathbf{i}}

\def\PPP{\mathcal{P}}
\def\TTT{\mathcal{T}}
\def\00{\mathbf{0}}
\def\nn{\mathbf{n}}
\def\NN{\mathbb{N}}
\def\RR{\mathbb{R}}
\def\ZZ{\mathbb{Z}}

\newtheorem{theorem}{Theorem}

\newtheorem{lemma}[theorem]{Lemma}
\newtheorem{corollary}[theorem]{Corollary}

\theoremstyle{definition}
\newtheorem{example}[theorem]{Example}

\theoremstyle{remark}

\title[Divided Differences of Multivariate Implicit Functions]{Divided Differences\\ of Multivariate Implicit Functions}
\author{Georg Muntingh}
\address{CMA / Department of Mathematics, University of Oslo, P.O. Box 1053, Blindern, N-0316, Oslo, Norway}
\email{georgmu@math.uio.no}

\begin{document}

\maketitle

\begin{abstract}
\noindent Under general conditions, the equation $g(x^1,\ldots,x^q,y) = 0$ implicitly defines $y$
locally as a function of $x^1,\ldots,x^q$. In this article, we express divided differences of
$y$ in terms of divided differences of $g$, generalizing a recent formula for the case where $y$ is univariate.
The formula involves a sum over a combinatorial structure whose elements can be viewed either as polygonal partitions
or as plane trees. Through this connection we prove as a corollary a formula for derivatives of $y$ in terms of derivatives of $g$.
\end{abstract}


\section{Introduction}\label{sec:Introduction}
Divided differences can be viewed as a discrete analogue of derivatives and are commonly used in approximation theory, see \cite{Boor05} for a survey.

Recently, Floater and Lyche introduced a multivariate chain rule for divided differences \cite{Floater10}, analogous to a multivariate form of Fa\'a di Bruno's formula for derivatives \cite{FaaDiBruno57, Johnson02, ConstantineSavits96}. In Theorem 1 in \cite{MuntinghFloater11}, this chain rule was applied to find an expression for divided differences of univariate implicit functions, thereby generalizing a formula by Floater and Lyche for divided differences of the inverse of a function \cite{FloaterLyche08}.

In Theorem \ref{thm:MainTheorem}, the Main Theorem of this paper, we generalize Theorem 1 in \cite{MuntinghFloater11} to divided differences of multivariate implicit functions. More precisely, for some open box $U\subset\RR^q$ and open interval $V\subset\RR$, let $y:U\longrightarrow V$ be a function that is implicitly defined by a function $g: U\times V\longrightarrow \RR$ via
\begin{equation}\label{eq:ImplicitlyDefined}
g\big(\xx,y(\xx)\big) = 0,\qquad
\frac{\partial g}{\partial y}\big(\xx,y(\xx)\big)\neq 0
\qquad \forall\ \xx\in U.
\end{equation}
Then the Main Theorem states that, for any rectangular grid
\[ \big\{x^1_0,\ldots,x^1_{n_1}\big\} \times \cdots \times \big\{x^q_0,\ldots,x^q_{n_q}\big\} \subset U,\]
we can express the divided difference $\big[x^1_0,\ldots,x^1_{n_1};\cdots;x^q_0,\ldots,x^q_{n_q} \big]y$ as a sum of terms involving the divided differences of $g$.

In the next section, we define these divided differences and explain our notation. In Section \ref{sec:RecursiveFormula}, we apply the multivariate chain rule to derive a formula that recursively expresses divided differences of $y$ in terms of divided differences of $g$ and lower-order divided differences of $y$. This recursive formula is ``solved'' in Section \ref{sec:Formula}, in the sense that it is used to derive a closed-form nonrecursive formula that expresses divided differences of $y$ solely in terms of divided differences of $g$. This formula is stated in the Main Theorem as a sum over polygonal partitions. It is shown in Section \ref{sec:PartitionsAndTrees}, that such polygonal partitions correspond to plane trees of a certain type, giving rise to an alternative form of the Main Theorem. Switching between these combinatorial structures, we are able to prove as a special case in Section \ref{sec:ImplicitDerivatives} a generalization of a formula by Comtet, Fiolet, and Wilde for the \emph{derivatives} of $y$ in terms of the derivatives of $g$.

\section{Divided differences}\label{sec:DividedDifferences}
Consider a function $y: U\longrightarrow \RR$ defined on some open box
\begin{equation}\label{eq:U}
U = (a_1,b_1)\times \cdots \times (a_q,b_q)\subset \RR^q.
\end{equation}
Suppose that, for some integers $n_1,\ldots, n_q \geq 0$ and all $j = 1,\ldots, q$, we are given points $x^j_0, \ldots, x^j_{n_j} \in (a_j, b_j)$ satisfying $a_j < x^j_0 < \cdots < x^j_{n_j} < b_j$. The Cartesian product 
\begin{equation}\label{eq:Grid}
\big\{ x^1_0, \ldots, x^1_{n_1} \big\} \times  \cdots \times \big\{ x^q_0, \ldots, x^q_{n_q} \big\}
\end{equation}
defines a \emph{[rectangular] grid} of points in $U$. The \emph{divided difference} of $y$ at this grid, denoted by
\begin{equation}\label{eq:mvdd}
 \big[x^1_0, \ldots, x^1_{n_1};\ \cdots ;x^q_0, \ldots, x^q_{n_q}\big]y,
\end{equation}
can be defined recursively as follows. If $n_1 = \cdots = n_q = 0$, the grid consists of only one point
$(x^1_0,\ldots,x^q_0)$, and we define $\big[x^1_0;\cdots ;x^q_0\big]y := y(x^1_0,\ldots,x^q_0)$
as the value of $y$ at this point. In case $n_j > 0$ for some $1\leq j\leq q$, we can define (\ref{eq:mvdd}) recursively by
\begin{equation}\label{eq:mvdd-2}
(x^j_{n_j} - x^j_0) \big[x^1_0, \ldots, x^1_{n_1};\ \cdots ;x^q_0, \ldots, x^q_{n_q}\big]y =
\end{equation}
\[ \qquad \big[x^1_0, \ldots, x^1_{n_1};\ \cdots ; \widehat{x^j_0}, x^j_1, \ldots, x^j_{n_j-1},          x^j_{n_j};\  \cdots ;x^q_0, \ldots, x^q_{n_q}\big]y - \]
\[ \qquad \big[x^1_0, \ldots, x^1_{n_1};\ \cdots ;          x^j_0 , x^j_1, \ldots, x^j_{n_j-1}, \widehat{x^j_{n_j}};\ \cdots ;x^q_0, \ldots, x^q_{n_q}\big]y,\ \, \]
where the hat signifies omission of a symbol.
If several of the $n_j$ are greater than zero, the divided difference (\ref{eq:mvdd}) is uniquely defined by any of these 
recursive formulas. We refer to the dimensions $(n_1, \ldots, n_q)$ of the grid as the \emph{order} of the divided difference in Equation \ref{eq:mvdd}.

For any $\aaa = (a_1,\ldots,a_q), \bb = (b_1,\ldots,b_q) \in \NN^q$, write $\aaa\leq \bb$ whenever $a_j\leq b_j$ for every $1\leq j\leq q$. Additionally, we write $\aaa < \bb$ whenever $\aaa\leq \bb$ and $\aaa\neq \bb$. In this manner, the symbol $\leq$ defines a partial order on $\NN^q$. We use the notation
\[ [\xx:\aaa,\bb]y := \big[x^1_{a_1},x^1_{a_1 + 1},\ldots,x^1_{b_1};\ \cdots\ ; x^q_{a_q},x^q_{a_q + 1},\ldots, x^q_{b_q}\big]y
\]
for the divided difference of $y$ with respect to the grid of all points with indices ``between $\aaa$ and $\bb$''.

Divided differences of the function $g:U\times V\longrightarrow \RR$ in Equation \ref{eq:ImplicitlyDefined} are defined similarly. For these functions, however, we stress the distinction between the variables $x^1,\ldots,x^q$ and the variable $y$ by replacing the final semi-colon by a bar in our notation.

As the notation of Equation \ref{eq:mvdd} quickly grows cumbersome, we shall more often than not shorten the notation for divided differences to one that involves just the indices,
\begin{equation}\label{eq:divdiff-short2}
\big[i_{1,0}\cdots i_{1,s_1} ;\ \cdots ;i_{q,0}\cdots i_{q,s_q}\big]y := \big[x^1_{i_{1,0}},\ldots,x^1_{i_{1,s_1}};\ \cdots; x^q_{i_{q,0}},\ldots,x^q_{i_{q,s_q}}\big]y,
\end{equation}
\begin{equation}\label{eq:divdiff-short3}
\big[i_{1,0}\cdots i_{1,s_1} ;\ \cdots ;i_{q,0}\cdots i_{q,s_q}\big|j_0\cdots j_t\big]g := 
\end{equation}
\[ \big[x^1_{i_{1,0}},\ldots,x^1_{i_{1,s_1}};\ \cdots; x^q_{i_{q,0}},\ldots,x^q_{i_{q,s_q}}\big|y_{j_0},\ldots,y_{j_t}\big]g. \]

We can let some of the points coalesce by taking limits, as long as $y$ is sufficiently smooth.
In particular, letting all points in the grid coalesce to a single point $\xx_0 = (x^1_0,\ldots,x^q_0)$ yields, for any tuple $\nn = (n_1,\ldots,n_q) \in \NN^q$,
\[
\big[\underbrace{x^1_0, \ldots, x^1_0}_{n_1 + 1};\ \cdots; \underbrace{x^q_0, \ldots, x^q_0}_{n_q + 1}\big]y =
\frac{1}{\nn!} \frac{\partial^{|\nn|} y} {\partial \xx^\nn}(\xx_0).
\]
Here the derivatives are written in multi-index notation, $|\nn| := n_1 + \cdots + n_q$, and $\nn! := n_1!\cdots n_q!$\,.
Letting, in addition, the $y$-values coalesce to a single point $y_0$ yields
\[
\big[\underbrace{x^1_0, \ldots, x^1_0}_{n_1 + 1};\ \cdots; \underbrace{x^q_0,\ldots,x^q_0}_{n_q + 1}\big| \underbrace{y_0, \ldots, y_0}_{m+1}\big]g =
\frac{1}{\nn! m!} \frac{\partial^{|\nn| + m} g} {\partial \xx^\nn \partial y^m} (\xx_0,y_0).
\]

\section{A recursive formula}\label{sec:RecursiveFormula}
Let $y$ and $g$ be related as in Equation \ref{eq:ImplicitlyDefined}. In this section, we derive a formula that expresses divided differences of $y$ recursively as divided differences of $g$ and lower-order divided differences of $y$.

Consider a composition of functions $\RR^q\stackrel{\ff}{\longrightarrow} U\times V\stackrel{g}{\longrightarrow} \RR$, where we write
$\ff:\xx\longmapsto \big(f^1(\xx), \ldots, f^q(\xx), f'(\xx)\big)$.
Here $f'$ does not denote the derivative of some function $f$ but is simply notation for the last component of $\ff$.
Let be given a nonzero tuple $\nn = (n_1,\ldots,n_q)\in \NN^q$ and a grid of points
\[
\big\{\xx_\ii : \00\leq \ii\leq \nn \big\} =
\big\{ x^1_0, \ldots, x^1_{n_1} \big\} \times  \cdots \times \big\{ x^q_0, \ldots, x^q_{n_q} \big\}
\subset U,
\]
where $x^j_0 \leq \cdots \leq x^j_{n_j}$ for $j=1,\ldots,q$. Note that we allow for these coordinates to coincide. Let $f'_{\ii} := f'(\xx_\ii)$ and $f^j_{\ii} := f^j(\xx_\ii)$ for $j = 1,\ldots,q$ and $\00\leq \ii\leq \nn$. From \cite[Theorem 2]{Floater10}, we have the following \emph{multivariate chain rule} in case $\ff$ and $g$ are sufficiently smooth,
\begin{equation}\label{eq:ChainRule} \!\!\!\![\xx : \00,\nn](g\circ \ff) =
\sum_{k=1}^{|\nn|}\ 
\sum_{\00 = \ii_0 < \cdots < \ii_k = \nn}\ 
\sum_{0=j_0\leq \cdots \leq j_{q+1} = k}
\end{equation}
\[ \!\!\!\!\big[f^1_{\ii_{j_0}},f^1_{\ii_{j_0 + 1}},\ldots,f^1_{\ii_{j_1}};\cdots;f^q_{\ii_{j_{q-1}}},f^q_{\ii_{j_{q-1}+1}},\ldots,f^q_{\ii_{j_q}} \big| f'_{\ii_{j_q}}, f'_{\ii_{j_q + 1}},\ldots,f'_{\ii_{j_{q+1}}}\big]g \]
\[ \times
\left(\prod_{r=1}^q \prod_{j = j_{r-1} + 1}^{j_r} [\xx : \ii_{j-1}, \ii_j]f^r \right)
\left(\prod_{j = j_q + 1}^{j_{q+1}} [\xx : \ii_{j-1}, \ii_j]f' \right),
\]
where an empty product is considered to be $1$. (The formula in \cite{Floater10} includes a term for $k=0$, but this term doesn't show up because we assumed $\nn\neq \00$.) For any integer $q\geq 1$, let us define a \emph{path [in $\ZZ^q$]} as a finite sequence $(\pp_0, \pp_1, \ldots, \pp_n)$ of points in $\ZZ^q$ for which $\pp_0 < \pp_1 < \cdots < \pp_n$. We found it helpful to think of $0 = j_0 \leq \cdots \leq j_{q+1} = k$ as indices along the path $\00 = \ii_0 < \cdots < \ii_k = \nn$; see Figure \ref{fig:LatticePath}.

\begin{figure}
\begin{center}
\includegraphics[scale=0.58]{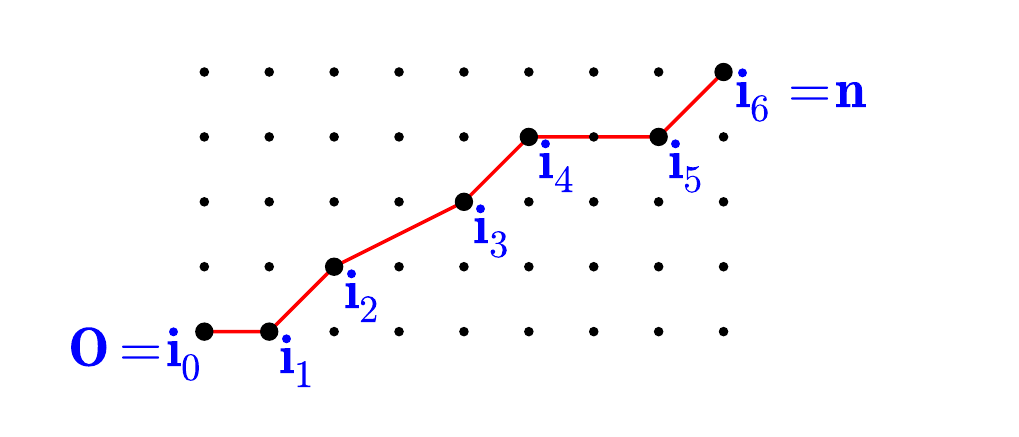}
\end{center}
\caption[A sequence representing a path]{Any sequence $\00 = \ii_0 < \ii_1 < \cdots < \ii_k = \nn$ represents a path from $\00$ to $\nn$.}
\label{fig:LatticePath}
\end{figure}

Next, let $\ff: \xx = (x^1,\ldots,x^q)\longmapsto \big(\xx,y(\xx)\big)$ define the graph of a function $y$ that is implicitly defined by $g$ as in Equation \ref{eq:ImplicitlyDefined}. Let $\{\ee_1,\ldots,\ee_q\}$ denote the standard basis of $\RR^q$, and let $1\leq j\leq k$ and $1\leq r\leq q$ be as in Equation \ref{eq:ChainRule}. It follows directly from Equation \ref{eq:mvdd-2} that the divided difference $[\xx:\ii_{j-1}, \ii_j]f^r$ of the coordinate function $f^r: (x^1,\ldots,x^q)\longmapsto x^r$ is equal to $1$ whenever $\ii_j - \ii_{j-1} = \ee_r$, and zero otherwise. The only choices of $0 = j_0 \leq \cdots \leq j_{q+1} = k$ that yield a nonzero term in Equation \ref{eq:ChainRule} are therefore those satisfying
\begin{equation}\label{eq:CompatibleSequences}
\ii_j - \ii_{j-1} = \ee_r,\quad \text{for}\quad j=j_{r-1}+1,\ldots,j_r,\quad r=1,\ldots,q.
\end{equation}

Alternatively, let $(s_1,\ldots,s_q, t) := (j_1 - j_0, \ldots, j_q - j_{q-1}, j_{q+1} - j_q)$ be the sequence
of \emph{jumps} in the sequence $(j_0,\ldots,j_{q+1})$. In terms of these jumps, Equation \ref{eq:CompatibleSequences} is equivalent to the statement that the path $\ii_0 < \cdots < \ii_k$ starts with
\begin{itemize}
\item[] $s_1$ steps of $\ee_1$, followed by
\item[] $s_2$ steps of $\ee_2$, followed by
\item[] $\ \vdots$ 
\item[] $s_q$ steps of $\ee_q$, followed by 
\item[] $t$ arbitrary steps.
\end{itemize}
Let us call any tuple $(s_1,\ldots,s_q,t)$ with this property \emph{compatible with $(\ii_0,\ldots,\ii_k)$}, or simply \emph{compatible} if it is clear which sequence $(\ii_0,\ldots,\ii_k)$ is referred to. Note that such a tuple forms an integer partition $k = s_1 + \cdots + s_q + t$. Equation \ref{eq:ChainRule} thus implies
\begin{equation}\label{eq:ChainRule2}
[\xx : \00,\nn] g\big(\cdot,y(\cdot)\big) =
\sum_{k=1}^{|\nn|}\ 
\sum_{\00 = \ii_0 < \cdots < \ii_k = \nn}\ 
\sum_{\substack{\text{compatible}\\ (s_1,\ldots,s_q,t)}}\ 
\end{equation}
\[ [0\,1\cdots s_1;\ \cdots\ ;0\,1\cdots s_q| \ii_{|\ss|}\,\ii_{|\ss|+1}\cdots \ii_{|\ss|+t}]g \prod_{j = |\ss| + 1}^{|\ss| + t} \![\xx: \ii_{j-1}, \ii_j]y, \]
where we used the shorthand notation
\[ \hspace{1.2em} [0\,1\cdots s_1;\ \cdots\ ;0\,1\cdots s_q| \ii_{|\ss|}\,\ii_{|\ss|+1}\cdots \ii_{|\ss|+t}]g\]
\[ = \big[x^1_0, x^1_1,\ldots, x^1_{s_1};\ \cdots\ ;x^q_0,x^q_1,\ldots,x^q_{s_q} \big| y_{\ii_{|\ss|}},y_{\ii_{|\ss|+1}},\ldots, y_{\ii_{|\ss|+t}}\big]g,\]
$y_{\ii} := y(\xx_\ii)$, from Equation \ref{eq:divdiff-short3}, for the divided differences of~$g$.

If $y$ is implicitly defined by $g$ as in Equation \ref{eq:ImplicitlyDefined}, the left hand side of Equation \ref{eq:ChainRule2} is zero. Suppose $\nn$ has length $|\nn| = 1$. Then $\nn = \ee_r$ for some $1\leq r\leq q$. In this case, the right hand side of Equation \ref{eq:ChainRule2} comprises two terms with $k=1$ and $\00 = \ii_0 < \ii_1 = \ee_r$. One finds
\[ 0 = [\xx: \00, \ee_r]g\big(\cdot,y(\cdot)\big) = \]
\[ [0;\ \cdots;0|\00\,\ee_r]g\, [\xx: \00,\ee_r]y +
          [\underbrace{0;\ \cdots;0}_{r-1};0\,1;0;\ \cdots;0|\ee_r]g,\]
or, equivalently,
\begin{equation}\label{eq:RecursionFormulaInitial}\tag{\textbf{R1}}
[\xx:\00,\ee_r]y =
 -\frac{[\overbrace{0;\ \cdots;0}^{r-1};0\,1;0;\ \cdots ;0|\ee_r]g}{[0;\ \cdots ;0|\00\,\ee_r]g},\qquad \text{for } r=1,\ldots,q.
\end{equation}
For example, when $y$ is a function of $q=2$ variables, this equation represents the two formulas
\begin{equation}\label{eq:DivDiff10}
[0\,1;0]y = -\frac{[0\,1;0|\ee_1]g}{[0;0|\00\, \ee_1]g},\qquad
[0;0\,1]y = -\frac{[0;0\,1|\ee_2]g}{[0;0|\00\, \ee_2]g}.
\end{equation}

Now suppose $\nn$ has length $|\nn| > 1$. There is only one term in the right hand side of Equation \ref{eq:ChainRule2} with $k=1$. This term is given by
\[ \00 = \ii_0 < \ii_1 = \nn,\qquad s_1 = \cdots = s_q = 0, \qquad t = 1\]
and involves the highest order divided difference of $y$. Isolating this divided difference yields a formula that recursively expresses divided differences of $y$ in terms of divided differences of $g$ and lower-order divided differences of $y$,
\begin{equation}\label{eq:RecursionFormulaSubsequent}\tag{\textbf{R2}}
[\xx : \00,\nn]y =
\sum_{k=2}^{|\nn|}\ 
\sum_{\00 = \ii_0 < \cdots < \ii_k = \nn}\ 
\sum_{\substack{\text{compatible}\\ (s_1,\ldots,s_q,t)}}
\end{equation}
\[ \left(-\frac
{[0\,1\cdots s_1;\ \cdots\ ;0\,1\cdots s_q| \ii_{|\ss|}\ii_{|\ss|+1} \cdots \ii_{|\ss|+t}]g}
{[0;\ \cdots; 0|\00\,\nn]g}\right)
 \prod_{j = |\ss| + 1}^{|\ss| + t} [\xx: \ii_{j-1}, \ii_j]y . \]

Let us simplify this formula. The product in Equation \ref{eq:RecursionFormulaSubsequent} can be split into two products
\begin{equation}\label{eq:TwoProducts}
\prod_{\substack{j = |\ss| + 1\\ |\ii_j - \ii_{j-1}|   =  1}}^{|\ss| + t} [\xx: \ii_{j-1}, \ii_j]y\ \cdot
\prod_{\substack{j = |\ss| + 1\\ |\ii_j - \ii_{j-1}| \geq 2}}^{|\ss| + t} [\xx: \ii_{j-1}, \ii_j]y.
\end{equation}
From Equation \ref{eq:RecursionFormulaInitial} it follows that each divided difference in the first product can be expressed as a quotient of divided differences of $g$. Our ultimate goal is to express the left hand side of Equation \ref{eq:RecursionFormulaSubsequent} solely in terms of divided differences of $g$. To achieve this, it seems natural to split the right hand side into a part that can directly be expressed in terms of divided differences of $g$ and a remaining part involving higher-order divided differences of $y$. The former part can be expressed by introducing, for every sequence
\[ (i_0^1,\ldots,i_0^q) = \ii_0 < (i_1^1,\ldots,i_1^q) = \ii_1 < \cdots < (i_k^1,\ldots,i_k^q) = \ii_k, \]
a symbol $\{\ii_0\cdots \ii_k\}g$ for the expression
\begin{equation}\label{eq:CurlyBrackets} \sum_{\substack{\text{compatible}\\ (s_1,\ldots,s_q,t)}}
\left(-\frac
{\big[i_0^1\cdots (i_0^1 + s_1);\ \cdots\ ; i_0^q\cdots (i_0^q + s_q)\big| \ii_{|\ss|} \cdots \ii_{|\ss|+t}\big]g}
{\big[i_0^1;\ \cdots; i_0^q \big|\ii_0\,\ii_k\big]g}
\right)
\end{equation}
\[ \times \prod_{r=1}^q \prod_{\substack{j = |\ss| + 1\\ \ii_j - \ii_{j-1} = \ee_r}}^{|\ss|+t}
\left( - \frac
{\big[i_{j-1}^1;\ \cdots;i_{j-1}^{r-1}; i_{j-1}^r\, i_j^r; i_j^{r+1};\ \cdots;i_j^q\big|\ii_j\big]g}
{\big[i_{j-1}^1;\ \cdots;i_{j-1}^q\big|\ii_{j-1}\,\ii_j\big]g}
\right)
\]
involving only divided differences of $g$. Whenever it is hard to separate visually the multi-indices $\ii_0,\ldots,\ii_k$, we write $\{\ii_0,\ldots,\ii_k\}g$ instead of $\{\ii_0\cdots\ii_k\}g$.

The divided differences $[\xx:\ii_{j-1},\ii_j]y$ that appear in the second product of Equation \ref{eq:TwoProducts} satisfy $|\ii_j - \ii_{j-1}|\geq 2$.
As Equation \ref{eq:CompatibleSequences} guarantees that this cannot happen for $j \leq j_q = |\ss|$, we might as well start the product of these remaining divided differences at $j=1$ instead of at $j = |\ss| + 1$. This has the advantage of making the expression independent of $|\ss|$. Equation \ref{eq:RecursionFormulaSubsequent} can thus be written in the concise form
\begin{equation}\label{eq:RecursionFormula}\tag{\textbf{R2$^\prime$}}
[\xx: \00, \nn]y = \sum_{k=2}^{|\nn|} \ \sum_{\00=\ii_0<\cdots <\ii_k = \nn}  \{\ii_0\cdots \ii_k\}g
\prod_{\substack{j = 1\\ |\ii_j - \ii_{j-1}| \geq 2}}^k [\xx: \ii_{j-1}, \ii_j]y.
\end{equation}

Equation \ref{eq:RecursionFormulaInitial} gives a formula for $[\xx: \00,\nn]y$ when $|\nn| = 1$. Let us consider Equation \ref{eq:RecursionFormula} for the case that $|\nn| = 2$. For such $\nn$, either $\nn = 2\ee_r$ with $1\leq r\leq q$, or $\nn = \ee_r + \ee_s$ with $1\leq r < s\leq q$. In the examples below we compute $[\xx: \00, \nn]y$ for these two cases, assuming $q=2$ to simplify notation. 

\begin{example}\label{ex:20}
Suppose $\nn = 2\ee_1$ (the case $\nn = 2\ee_2$ is similar). The only possible path $\00=\ii_0 < \cdots < \ii_k = \nn$ with $k=2$ in Equation \ref{eq:RecursionFormula} is given by $\00 < \ee_1 < 2\ee_1$. As for such a path the product in Equation \ref{eq:RecursionFormula} is empty, one has $[\xx: \00, 2\ee_1]y = \{\00,\ee_1,2\ee_1\}g$.

To compute $\{\00,\ee_1,2\ee_1\}g$, we need to find out which integer partitions $2 = s_1 + s_2 + t$ are compatible with this path. These are precisely the triples $(s_1,s_2,t)$ for which
\begin{equation}\label{eq:CompatibleSequences2A}
\ee_1 = \ii_1 - \ii_0 = \ii_2 - \ii_1 = \cdots = \ii_{s_1} - \ii_{s_1 - 1},
\end{equation}
\begin{equation}\label{eq:CompatibleSequences2B}
\ee_2 = \ii_{s_1 + 1} - \ii_{s_1} = \ii_{s_1 + 2} - \ii_{s_1 + 1} = \cdots = \ii_{s_1 + s_2} - \ii_{s_1 + s_2 - 1},
\end{equation}
where the first (respectively second) statement is considered to be trivially satisfied whenever $s_1 = 0$ (respectively $s_2 = 0$).
As both $\ii_1 - \ii_0$ and $\ii_2 - \ii_1$ are equal to $\ee_1$, the first condition is automatically satisfied. As neither $\ii_1 - \ii_0$ nor $\ii_2 - \ii_1$ is equal to $\ee_2$, necessarily $s_2 = 0$. It follows that there are three triples $(s_1,s_2,t) = (0,0,2),\, (1,0,1),\, (2,0,0)$ compatible with $(\00,\ee_1,2\ee_1)$. Each of these sequences corresponds to a term in $\{\00,\ee_1,2\ee_1\}g$, and we conclude that
\begin{equation}\label{eq:DivDiff20}
[\xx: \00,2\ee_1]y = \{\00,\ee_1,2\ee_1\}g =
\end{equation}
\[ \quad -
\frac{[0;0|\00,\ee_1,2\ee_1]g}{[0;0|\00,2\ee_1]g}
\frac{[0\,1;0|\ee_1]g}{[0;0|\00,\ee_1]g}
\frac{[1\,2;0|2\ee_1]g}{[1;0|\ee_1,2\ee_1]g} \]
\[ \quad +
\frac{[0\,1;0|\ee_1,2\ee_1]g}{[0;0|\00,2\ee_1]g}
\frac{[1\,2;0|2\ee_1]g}{[1;0|\ee_1,2\ee_1]g} \]
\[ \quad - \frac{[0\,1\,2;0|2\ee_1]g}{[0;0|\00,2\ee_1]g}. \]
\end{example}

\begin{example}\label{ex:11} Suppose $\nn = \ee_1 + \ee_2$. Equation \ref{eq:RecursionFormula} is a sum over the two possible paths $\00 < \ee_1 < \ee_1 + \ee_2$ and $\00 < \ee_2 < \ee_1 + \ee_2$.

A triple $(s_1,s_2,t)$ is compatible with the path $\00 < \ee_1 < \ee_1 + \ee_2$ precisely when Equations \ref{eq:CompatibleSequences2A} and \ref{eq:CompatibleSequences2B} hold. For this path, the first equation is equivalent to $s_1$ either being $0$ or $1$. If $s_1 = 0$, then the second equation implies that $s_2 = 0$. If $s_1 = 1$, on the other hand, the second equation implies that $s_2$ is either $0$ or $1$. One finds three triples $(s_1,s_2,t) = (0,0,2),\, (1,0,1),\,(1,1,0)$ compatible with $(\00, \ee_1, \ee_1 + \ee_2)$, yielding
\[ \{\00, \ee_1, \ee_1 + \ee_2\}g = \]
\[ \quad -
\frac{[0;0|\00,\ee_1,\ee_1+\ee_2]g}{[0;0|\00,\ee_1+\ee_2]g}
\frac{[0\,1;0|\ee_1]g}{[0;0|\00,\ee_1]g}
\frac{[1;0\,1|\ee_1+\ee_2]g}{[1;0|\ee_1,\ee_1+\ee_2]g} \]
\[ \quad +
\frac{[0\,1;\ee_1|\ee_1,\ee_1+\ee_2]g}{[0;0|\00,\ee_1 + \ee_2]g}
\frac{[1;0\,1|\ee_1+\ee_2]g}{[1;0|\ee_1,\ee_1+\ee_2]g} \]
\[ \quad - \frac{[0\, 1;0\,1|\ee_1+\ee_2]g}{[0; 0|\00,\ee_1 + \ee_2]g}.\]

Similarly, a triple $(s_1,s_2,t)$ is compatible with the path $\00 < \ee_2 < \ee_1 + \ee_2$ precisely when Equations \ref{eq:CompatibleSequences2A}, \ref{eq:CompatibleSequences2B} hold. For this path, however, the fact that $\ii_2 - \ii_1 = \ee_1$ comes after $\ii_1 - \ii_0 = \ee_2$ implies that $s_1 = 0$. One finds two triples $(s_1,s_2,t) = (0,0,2),\,(0,1,1)$ compatible with $(\00, \ee_2, \ee_1 + \ee_2)$, yielding
\[ \{\00,\ee_2,\ee_1+\ee_2\}g = \]
\[ \quad -
\frac{[0;0|\00,\ee_2,\ee_1+\ee_2]g}{[0;0|\00,\ee_1+\ee_2]g}
\frac{[0;0\,1|\ee_2]g}{[0;0|\00,\ee_2]g}
\frac{[0\,1;1|\ee_1+\ee_2]g}{[0;1|\ee_2,\ee_1+\ee_2]g} \]
\[ \quad +
\frac{[0;0\,1|\ee_2,\ee_1+\ee_2]g}{[0;0|\00,\ee_1 + \ee_2]g}
\frac{[0\,1;1|\ee_1+\ee_2]g}{[0;1|\ee_2,\ee_1+\ee_2]g}.\]
As for both paths the product in Equation \ref{eq:RecursionFormula} is empty, it follows that
\begin{equation}\label{eq:DivDiff11}
[\xx:\00,\ee_1+\ee_2]y = \{\00, \ee_1, \ee_1+\ee_2\}g + \{\00, \ee_2, \ee_1+\ee_2\}g.
\end{equation} 
\end{example}

\section{A formula for divided differences of implicit functions}\label{sec:Formula}
Let $y$ be implicitly defined by $g$ as in Equation \ref{eq:ImplicitlyDefined}. In this section we derive a formula that expresses divided differences of $y$ solely in terms of divided differences of $g$. For $\nn$ with $|\nn| = 2$, Equation \ref{eq:RecursionFormula} immediately yields the two formulas
\begin{equation}\label{eq:2}  [\xx: \00, 2\ee_r]y = \{\00,\ee_r,2\ee_r\}g, \end{equation}
\begin{equation}\label{eq:11} [\xx: \00, \ee_r+\ee_s]y = \{\00, \ee_r, \ee_r + \ee_s\}g + \{\00, \ee_s, \ee_r + \ee_s\}g, \end{equation}
where $1\leq r < s \leq q$ and the expressions $\{\ii_0\cdots\ii_k\}$ are defined in Equation \ref{eq:CurlyBrackets}. For $\nn$ with $|\nn| = 3$, one can distinguish three cases:  $\nn = 3\ee_r$, $\nn = 2\ee_r + \ee_s$, and $\nn = \ee_r + \ee_s + \ee_t$, with $1\leq r, s, t\leq q$ distinct. Let us compute $[\xx: \00, \nn]$ for these $\nn$ to get a feel for what a general formula should be. Repeatedly applying Equation \ref{eq:RecursionFormula} yields
\begin{equation}\label{eq:3} [\xx: \00, 3\ee_r]y = \end{equation}
\[ \phantom{+} \{\00,\ee_r,2\ee_r,3\ee_r\}g + \{\00,2\ee_r,3\ee_r\}g\cdot \{\00,\ee_r,2\ee_r\}g\]
\[  + \{\00,\ee_r,3\ee_r\}g \cdot \{\ee_r,2\ee_r,3\ee_r\}g \]
\vspace{-3mm}
\begin{equation}\label{eq:21} [\xx: \00, 2\ee_r + \ee_s]y = \end{equation}
\[ \phantom{+} \{\00,\ee_r,2\ee_r,2\ee_r+\ee_s\}g + \{\00,2\ee_r,2\ee_r+\ee_s\}g\cdot \{\00,\ee_r,2\ee_r\}g \]
\[                                          + \{\00,\ee_r, 2\ee_r+\ee_s\}g\cdot \{\ee_r,2\ee_r,2\ee_r+\ee_s\}g\]
\[          + \{\00,\ee_r,\ee_r + \ee_s, 2\ee_r+\ee_s\}g + \{\00,\ee_r+\ee_s,2\ee_r+\ee_s\}g\cdot \{\00,\ee_r,\ee_r+\ee_s\}g \]
\[                                                 + \{\00,\ee_r, 2\ee_r+\ee_s\}g\cdot \{\ee_r,\ee_r+\ee_s,2\ee_r+\ee_s\}g \]
\[          + \{\00,\ee_s,\ee_r + \ee_s, 2\ee_r+\ee_s\}g + \{\00,\ee_r+\ee_s,2\ee_r+\ee_s\}g\cdot \{\00,\ee_s,\ee_r+\ee_s\}g \]
\[                                                 + \{\00,\ee_s, 2\ee_r+\ee_s\}g\cdot \{\ee_s,\ee_r+\ee_s,2\ee_r+\ee_s\}g \]
\vspace{-2mm}
\begin{equation}\label{eq:111} [\xx: \00, \ee_r + \ee_s + \ee_t]y = \end{equation}
\[ \phantom{+} \{\00,\ee_r,\ee_r+\ee_s,\ee_r+\ee_s+\ee_t\}g + \{\00,\ee_r+\ee_s,\ee_r+\ee_s+\ee_t\}g\cdot \{\00,\ee_r,\ee_r+\ee_s\}g \]
\[  + \{\00,\ee_r,\ee_r+\ee_s+\ee_t\}g\cdot \{\ee_r,\ee_r+\ee_s,\ee_r+\ee_s+\ee_t\}g \]
\[        + \{\00,\ee_r,\ee_r+\ee_t,\ee_r+\ee_s+\ee_t\}g + \{\00,\ee_r+\ee_t,\ee_r+\ee_s+\ee_t\}g\cdot \{\00,\ee_r,\ee_r+\ee_t\}g\]
\[  + \{\00,\ee_r,\ee_r+\ee_s+\ee_t\}g\cdot \{\ee_r,\ee_r+\ee_t,\ee_r+\ee_s+\ee_t\}g\] 
\[        + \{\00,\ee_s,\ee_r+\ee_s,\ee_r+\ee_s+\ee_t\}g + \{\00,\ee_r+\ee_s,\ee_r+\ee_s+\ee_t\}g\cdot \{\00,\ee_s,\ee_r+\ee_s\}g\]
\[  + \{\00,\ee_s,\ee_r+\ee_s+\ee_t\}g\cdot \{\ee_s,\ee_r+\ee_s,\ee_r+\ee_s+\ee_t\}g\]
\[        + \{\00,\ee_s,\ee_s+\ee_t,\ee_r+\ee_s+\ee_t\}g + \{\00,\ee_s+\ee_t,\ee_r+\ee_s+\ee_t\}g\cdot \{\00,\ee_s,\ee_s+\ee_t\}g\]
\[  + \{\00,\ee_s,\ee_r+\ee_s+\ee_t\}g\cdot \{\ee_s,\ee_s+\ee_t,\ee_r+\ee_s+\ee_t\}g\]
\[        + \{\00,\ee_t,\ee_r+\ee_t,\ee_r+\ee_s+\ee_t\}g + \{\00,\ee_r+\ee_t,\ee_r+\ee_s+\ee_t\}g\cdot \{\00,\ee_t,\ee_r+\ee_t\}g\]
\[  + \{\00,\ee_t,\ee_r+\ee_s+\ee_t\}g\cdot \{\ee_t,\ee_r+\ee_t,\ee_r+\ee_s+\ee_t\}g\]
\[        + \{\00,\ee_t,\ee_s+\ee_t,\ee_r+\ee_s+\ee_t\}g + \{\00,\ee_s+\ee_t,\ee_r+\ee_s+\ee_t\}g\cdot \{\00,\ee_t,\ee_s+\ee_t\}g\]
\[  + \{\00,\ee_t,\ee_r+\ee_s+\ee_t\}g\cdot \{\ee_t,\ee_s+\ee_t,\ee_r+\ee_s+\ee_t\}g\]
These three formulas exhibit a remarkable pattern. For every choice of the path $\00 = \pp_0 < \pp_1 < \pp_2 < \pp_3 = \nn$, we seem to be getting a sum
\[ \{\pp_0 \pp_1 \pp_2 \pp_3\}g +
   \{\pp_0 \pp_2 \pp_3\}g\cdot \{\pp_0 \pp_1 \pp_2\}g + \{\pp_0 \pp_1 \pp_3\}g\cdot \{\pp_1 \pp_2 \pp_3\}g. \]
This expression bears a striking resemblance to the right hand side of the univariate formula
\[ [0123]y = \{0123\}g + \{023\}g\, \{012\}g + \{013\}g\, \{123\}g\]
established in \cite[Theorem 4]{MuntinghFloater11}. This suggests that, for general $\nn = (n_1,\ldots,n_q)$, the divided difference $[\xx:\00,\nn]y$ is a sum of ${n_1+\cdots +n_q\choose n_1,\ldots,n_q}$ univariate formulas, one for each choice of the path $\00 = \pp_0 < \cdots < \pp_{|\nn|} = \nn$. See Figure \ref{fig:PathPolygonA} for an example of such a path.

Theorem \ref{thm:MainTheorem} casts this suspicion into a precise form. In order to state this Theorem, we introduce some notation for polygon partitions. With a sequence of labels $\pp_0,\pp_1,\ldots,\pp_n$ we associate the ordered vertices of a convex polygon. A \emph{partition of a convex polygon} is the result of connecting certain pairs of nonadjacent vertices with straight line segments, none of which intersect. We denote the set of all partitions of the polygon with vertices $\pp_0,\pp_1,\ldots,\pp_n$ by $\PPP(\pp_0,\pp_1,\ldots,\pp_n)$. The points $\pp_0,\pp_1,\ldots,\pp_n$ and line segments (between either adjacent or nonadjacent vertices) form the vertices and edges of a plane graph. As such, every partition $\pi\in \PPP(\pp_0,\pp_1,\ldots,\pp_n)$ is described by its set $F(\pi)$ of \emph{[oriented] faces}, which does not include the unbounded face. Each face $f\in F(\pi)$ is represented by a subsequence $f = (\vv_0,\vv_1,\ldots,\vv_k)$ of the sequence $(\pp_0,\pp_1,\ldots,\pp_n)$ of length at least three. We let $E(\pi)$ denote the set of \emph{edges} in $\pi$, each of which is represented by a subsequence $(\vv_0,\vv_1)$ of $(\pp_0,\pp_1,\ldots,\pp_n)$ of length two. An edge $(\vv_0,\vv_1)$ is called an \emph{outer edge} if the vertices $\vv_0,\vv_1$ are adjacent. If $\vv_0,\vv_1$ are not adjacent, $(\vv_0,\vv_1)$ is called an \emph{inner edge}. Figure \ref{fig:PathPolygonB} depicts an example of such a partition of a convex polygon.

Armed with this notation for partitions of convex polygons, we are now able to state the Main Theorem of this paper.

\begin{figure}
\begin{center}
\subfloat[]{\includegraphics[scale=0.58]{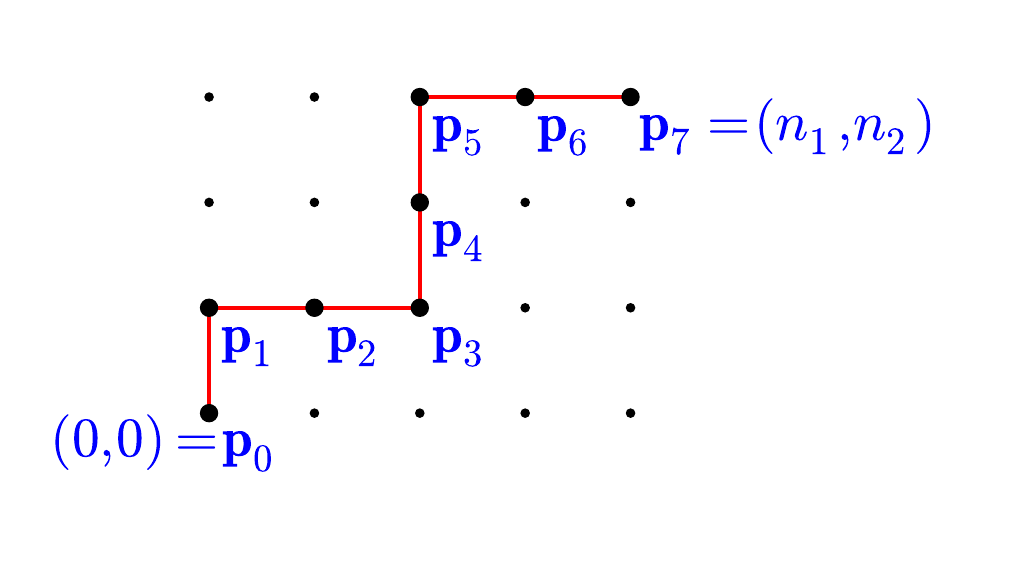}\label{fig:PathPolygonA}}
\subfloat[]{\includegraphics[scale=0.58]{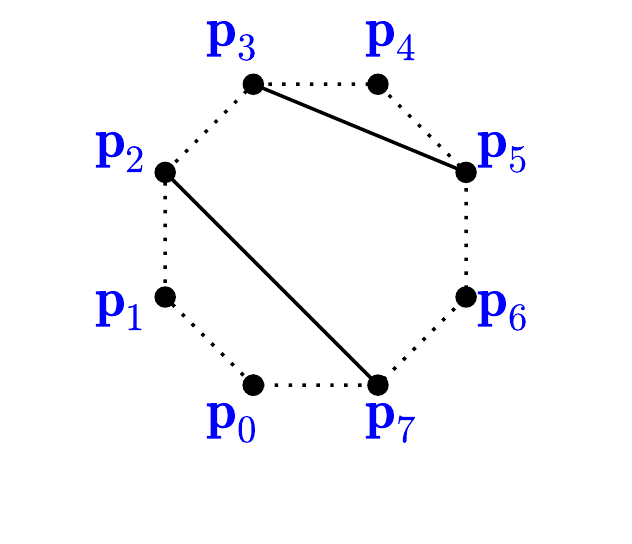}\label{fig:PathPolygonB}}
\end{center}
\caption[A path and a corresponding polygon partition]{For $\nn = (n_1,n_2) = (4,3)$, the figure to the left depicts a choice of a path $\00 = \pp_0 < \cdots < \pp_{|\nn|} = \nn$. The figure to the right shows a partition of the convex polygon corresponding to this path with faces
$(\pp_0, \pp_1, \pp_2, \pp_7), (\pp_2,\pp_3,\pp_5,\pp_6,\pp_7), (\pp_3,\pp_4,\pp_5)$, inner edges $(\pp_2,\pp_7), (\pp_3,\pp_5)$ (drawn solid), and outer edges $(\pp_0,\pp_1), (\pp_1,\pp_2), \ldots, (\pp_6,\pp_7), (\pp_0,\pp_7)$ (drawn dotted).}
\label{fig:PathPolygon}
\end{figure}

\begin{theorem}[Main Theorem]\label{thm:MainTheorem}
For $\nn$ with $|\nn|\geq 2$,
\begin{equation}\label{eq:MainTheorem}
[\xx: \00, \nn]y =
\sum_{\00 = \pp_0 < \pp_1 < \cdots < \pp_{|\nn|} = \nn}\ 
\sum_{\pi \in \PPP(\pp_0,\pp_1,\ldots,\pp_{|\nn|})}\ 
\prod_{(\vv_0,\vv_1,\ldots,\vv_r)\in F(\pi)}\{\vv_0 \vv_1\cdots \vv_r\}g.
\end{equation}
\end{theorem}

To prove Theorem \ref{thm:MainTheorem}, our plan is to use Equation \ref{eq:RecursionFormula} recursively to express $[\xx:\00,\nn]y$ solely in terms of divided differences of $g$. Before we proceed with this proof we assign some visual meaning to Equation \ref{eq:RecursionFormula} to highlight the backbone of this proof. We call a sequence $\ii = (\ii_0, \ii_1,\ldots, \ii_k)$ a \emph{subpath} of $\pp = (\pp_0,\pp_1,\ldots,\pp_n)$ and $\pp$ a \emph{superpath} of $\ii$, whenever
\[ \ii_0 = \pp_{l_0} < \pp_{l_0 + 1} < \cdots < \ii_1 = \pp_{l_1} < \pp_{l_1 + 1} < \cdots < \ii_k = \pp_{l_k},\]
for some increasing indices $0 = l_0 < l_1 < \cdots < l_k = n$. Every subpath $\ii$ of $\pp$ induces a partition in $\PPP(\pp_0, \pp_1,\ldots,\pp_n)$ whose set of faces comprises an \emph{inner face} $(\ii_0,\ii_1,\ldots,\ii_k)$ and \emph{outer faces} $(\ii_{j-1}, \ldots, \ii_j) = (\pp_{l_{j-1}}, \pp_{l_{j-1}+ 1},\ldots, \pp_{l_j})$ for every $j = 1,\ldots, k$ with $|\ii_j - \ii_{j-1}| \geq 2$. See Figure \ref{fig:InnerOuterB} for an example.

In general, a sequence $(\ii_0,\ii_1,\ldots,\ii_k)$ has several superpaths $(\pp_0,\pp_1,\ldots,\pp_n)$. Let us introduce some notation to consider simultaneously the partitions of the outer faces (each of which is a convex polygon itself) of $\ii = (\ii_0,\ii_1,\ldots,\ii_k)$ for all these superpaths. We define
\[ \PPP_\ii :=
\prod_{\substack{j = 1\\ m := |\ii_j - \ii_{j-1}| \geq 2}}^k ~ \coprod_{\ii_{j-1} = \qq_0 < \cdots < \qq_m = \ii_j} \PPP(\qq_0,\qq_1,\ldots,\qq_m), \]
which represents a set of tuples of partitions, each entry in such a tuple corresponding to a partition of a path with steps in $\{\ee_1,\ldots,\ee_q\}$ from $\ii_{j-1}$ to $\ii_j$ for some $j$. For example, for $\ii = (\ii_0,\ii_1,\ldots,\ii_5) = \big( (0,0), (0,1), (1,1), (2,2), (2,3),$ $(4,3) \big)$,
one has $|\ii_j - \ii_{j-1}| \geq 2$ only for $j = 3, 5$ (see Figure \ref{fig:InnerOuterA}). There are two paths with steps in $\{\ee_1,\ee_2\}$ from $\ii_2 = (1,1)$ to $\ii_3 = (2,2)$ and only one from $\ii_4 = (2,3)$ to $\ii_5 = (4,3)$. It follows that
\begin{align*}
\PPP_\ii = &\,   \Big( \PPP\big( \ii_2, (1,2), \ii_3 \big) \sqcup \PPP\big( \ii_2, (2,1), \ii_3 \big) \Big) \times \PPP\big( \ii_4, (3,3), \ii_5 \big)\\
         = &\,   \left\{ \Big( \big( \ii_2, (1,2), \ii_3 \big), \big(\ii_4, (3,3), \ii_5\big) \Big),
                         \Big( \big( \ii_2, (2,1), \ii_3 \big), \big(\ii_4, (3,3), \ii_5\big) \Big) \right\}.
\end{align*}

We now associate divided differences to these geometric objects. To each outer face $(\ii_{j-1},\ldots, \ii_j)$ we associate the divided difference $[\xx: \ii_{j-1}, \ii_j]y$, and to each inner face $(\ii_0, \ii_1, \ldots, \ii_k)$ we associate the expression $\{\ii_0\cdots \ii_k\}g$. For any sequence $\ii$ that appears in the sum of Equation \ref{eq:RecursionFormula}, the corresponding inner face therefore represents that part of Equation \ref{eq:RecursionFormula} that can be written solely in terms of divided differences of $g$, while the outer faces represent the part that is still expressed as a divided difference of $y$. 

\begin{figure}
\begin{center}
\subfloat[]{\includegraphics[scale=0.58]{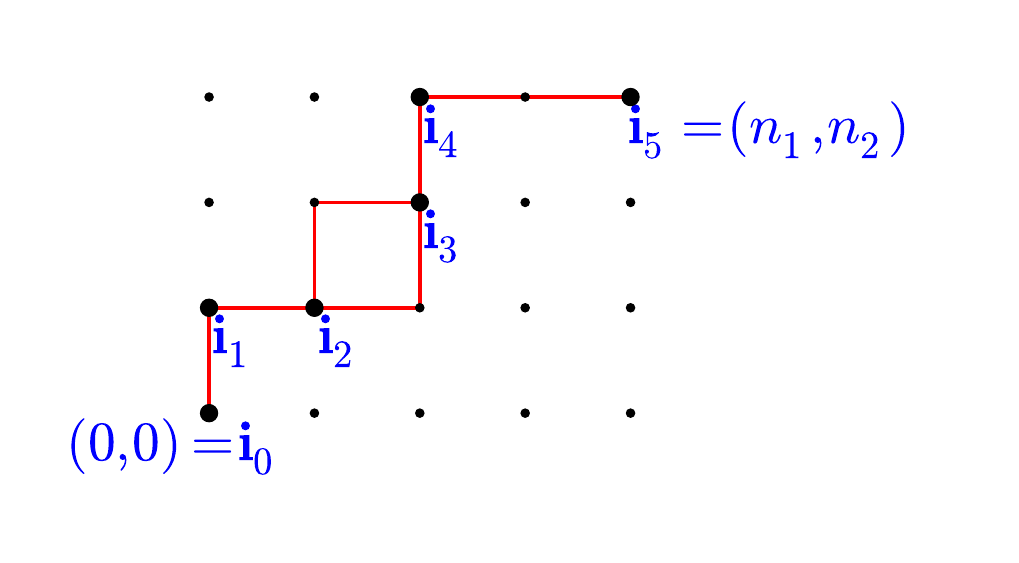}\label{fig:InnerOuterA}}
\subfloat[]{\includegraphics[scale=0.58]{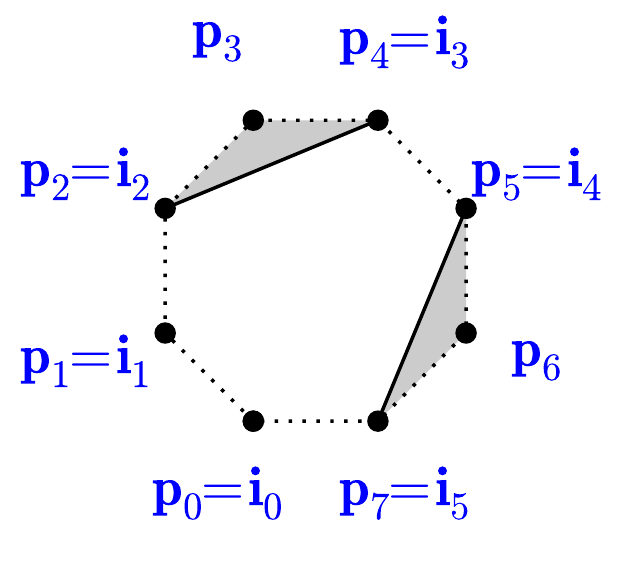}\label{fig:InnerOuterB}}
\end{center}
\caption[A sequence, its superpaths, and the corresponding polygon partition]{For $\nn = (n_1,n_2) = (4,3)$, the figure to the left shows the points in the sequence $\ii = (\ii_0,\ii_1,\ldots,\ii_5) = \big( (0,0), (0,1), (1,1), (2,2), (2,3), (4,3) \big)$, together with the paths traced out by its two superpaths $\pp = (\pp_0, \pp_1, \ldots, \pp_7)$. The figure to the right shows the convex polygon corresponding to each of these paths. The sequence $\ii$ gives rise to two outer faces $(\pp_2,\pp_3,\pp_4)$ and $(\pp_5,\pp_6,\pp_7)$, which are drawn shaded in the figure. Depending on the choice of the superpath $\pp$, the former outer face is either $\big(\ii_2, (1,2),\ii_3\big)$ or $\big(\ii_2, (2,1),\ii_3\big)$, while the latter is equal to $\big( \ii_4, (3,3), \ii_5\big)$ for both paths $\pp$.}
\label{fig:InnerOuter}
\end{figure}

\begin{proof}[Proof of Theorem \ref{thm:MainTheorem}] The proof is by induction on $|\nn|$. Equations \ref{eq:2}--\ref{eq:111} show that the formula holds for $|\nn| = 2, 3$. For a fixed $|\nn|\geq 4$, suppose the formula holds for all smaller $|\nn|$ (but with $|\nn|\geq 2$). Consider the recursion formula from Equation \ref{eq:RecursionFormula}. As in each term $k\geq 2$, one has $|\ii_j - \ii_{j-1}| < |\nn|$ for $j = 1,\ldots, k$. By induction, therefore, we can replace each divided difference $[\xx: \ii_{j-1}, \ii_j]y$ in Equation \ref{eq:RecursionFormula} by an expression involving only divided differences of $g$. The symbol $\PPP_\ii$ enables us to consider these expressions for $[\xx: \ii_{j-1}, \ii_j]y$ simultaneously, yielding
\begin{align}
  &\ \prod_{\substack{j = 1\\|\ii_j - \ii_{j-1}| \geq 2}}^k [\xx: \ii_{j-1}, \ii_j]y\label{eq:ProdY} \\
= &\ \prod_{\substack{j = 1\\m := |\ii_j - \ii_{j-1}| \geq 2}}^k \sum_{\ii_{j-1} = \qq_0 < \cdots < \qq_m = \ii_j}\ \sum_{\pi \in \PPP(\qq_0,\ldots,\qq_m)}\ \prod_{(\vv_0,\ldots,\vv_r)\in F(\pi)}\{\vv_0\cdots \vv_r\}g \notag \\
= &\ \sum_{(\pi_1,\pi_2,\ldots )\in \PPP_\ii} \prod_{\substack{j\geq 1\\ (\vv_0,\ldots,\vv_r)\in F(\pi_j)}} \{\vv_0\cdots \vv_r\}g. \notag
\end{align}

For a given sequence $\00 = \ii_0 < \cdots < \ii_k = \nn$ with $k\geq 2$, the set $\PPP_\ii$ can be identified with the set
\[ \left\{ \pi\in \PPP(\pp_0,\ldots,\pp_{|\nn|}):\00 = \pp_0 < \cdots < \pp_{|\nn|} = \nn,\ \ii\in F(\pi)\right\} \]
by the bijection that maps any tuple $(\pi_1,\pi_2,\ldots)$ in $\PPP_\ii$ to the partition $\pi$ with $F(\pi) = \{\ii\} \cup F(\pi_1) \cup F(\pi_2) \cup \cdots$.
Applying this bijection to Equation \ref{eq:ProdY} and substituting the result into the recursive formula from Equation \ref{eq:RecursionFormula} yields
\begin{align}
  &\ [\xx: \00, \nn]y \notag \\
= &\ \sum_{k=2}^{|\nn|}\ \sum_{\00=\ii_0<\cdots <\ii_k = \nn}  \{\ii_0\cdots \ii_k\}g \sum_{(\pi_1,\pi_2,\ldots )\in \PPP_\ii} \prod_{\substack{j\geq 1\\ (\vv_0,\ldots,\vv_r)\in F(\pi_j)}} \!\!\!\{\vv_0\cdots \vv_r\}g \notag \\
= &\ \sum_{k=2}^{|\nn|} \sum_{\substack{\00=\ii_0<\cdots <\ii_k = \nn\\\text{a subpath of}\\ \00 = \pp_0 < \cdots < \pp_{|\nn|} = \nn}} \sum_{\substack{\pi\in \PPP(\pp_0,\ldots,\pp_{|\nn|})\\ \ii\in F(\pi)}}\ \prod_{(\vv_0,\ldots,\vv_r)\in F(\pi)} \{\vv_0\cdots \vv_r\}g \notag \\
= &\ \sum_{\00 = \pp_0 < \cdots < \pp_{|\nn|} = \nn}\ \sum_{\pi\in \PPP(\pp_0,\ldots,\pp_{|\nn|})}\ \prod_{(\vv_0,\ldots,\vv_r)\in F(\pi)} \{\vv_0\cdots \vv_r\}g. \notag\qquad \qedhere
\end{align}
\end{proof}

\section{Polygon partitions and plane trees}\label{sec:PartitionsAndTrees}
While the compact nature of Equation \ref{eq:MainTheorem} is useful to state and prove Theorem \ref{thm:MainTheorem}, it is less appropriate for finding a specialized formula for derivatives. In this section we adapt Equation \ref{eq:MainTheorem} to a form better suited for this purpose.

Following \cite[p. 294]{Stanley97}, one defines a \emph{plane tree} $T$ recursively as a finite set of vertices such that one specially designated vertex is called the \emph{root}, and the remaining vertices are put into an \emph{ordered} partition $(T_1,\ldots, T_m)$ of $m\geq 0$ disjoint nonempty sets $T_1,\ldots,T_m$, each of which is a plane tree. We recall the following Lemma, which appears as Proposition 6.2.1 in \cite{Stanley99}.
\begin{lemma}\label{lem:DissectionsPlanarTrees} For all integers $m,n$ with $m > n\geq 2$, there is a bijection between the following two structures:
\begin{itemize}
\item Plane trees with $m$ vertices of which $n$ are leaves, and all other vertices have at least two descendants.
\item Partitions with $m - n$ faces of a convex polygon with $n + 1$ vertices.
\end{itemize}
\end{lemma}

We now explicitly describe this bijection. Suppose we are given a polygon partition in $\PPP(\pp_0,\pp_1,\ldots,\pp_n)$ with $m-n$ faces. To the edge $(\pp_0, \pp_n)$ we associate a vertex that represents the root of our tree. As $(\pp_0,\pp_n)$ is an outer edge, it belongs to a unique face. The other edges of this face are taken to be the descendants of the root vertex. As we are constructing a \emph{plane} tree we need to order these descendants; the vertices correspond from left to right to the edges encountered when traversing the border of the face clockwise, starting at $(\pp_0,\pp_n)$. We then repeat this process for each of the new edges until we are out of edges. This construction yields a plane tree with $n$ leaves corresponding to the outer edges $(\pp_0,\pp_1),(\pp_1,\pp_2),\ldots,(\pp_{n-1},\pp_n)$ and $m-n$ nonleaf vertices corresponding to the faces of the polygon partition.

For example, for the partition in $\PPP(\pp_0,\pp_1,\ldots,\pp_7)$ with set of faces
\[ \big\{(\pp_0,\pp_5,\pp_6,\pp_7), (\pp_0,\pp_2,\pp_5), (\pp_0,\pp_1,\pp_2), (\pp_2,\pp_4,\pp_5), (\pp_2,\pp_3,\pp_4)\big\},\]
this bijection can be visualized as follows.

\includegraphics[scale=0.45]{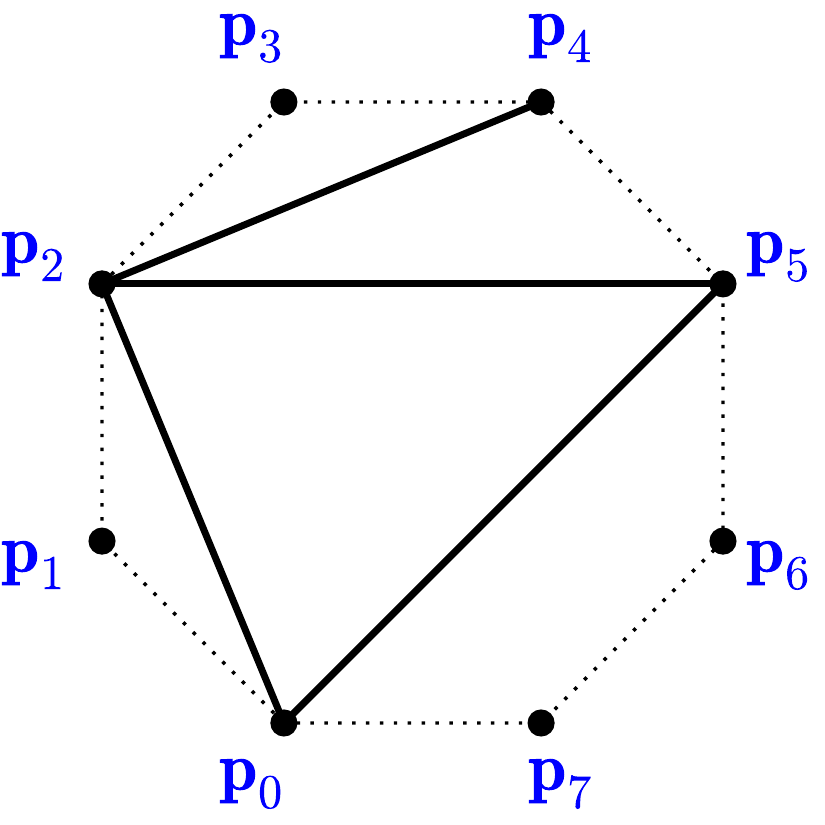}\qquad \includegraphics[scale=0.45]{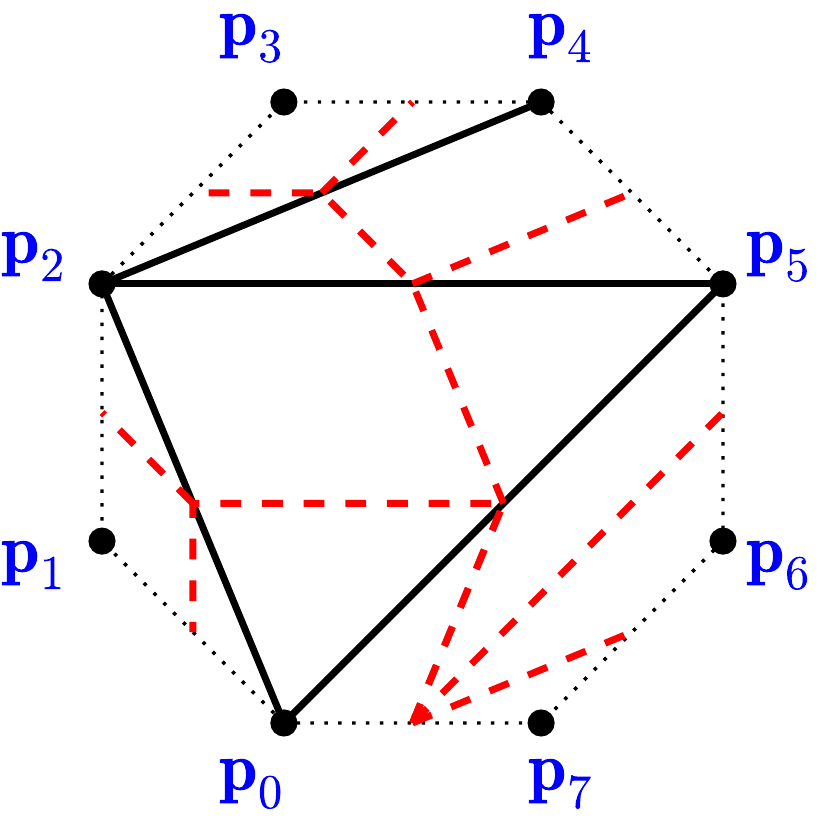}\quad \includegraphics[scale=0.45]{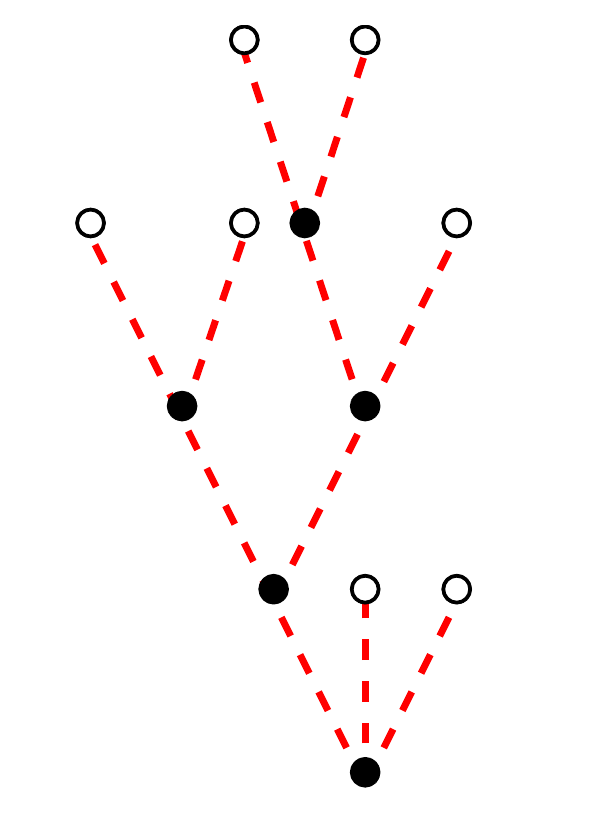}

\noindent Here leaves are drawn as circles and nonleaf vertices are drawn as discs.

Conversely, suppose we are given a plane tree with $m$ vertices of which $n$ are leaves, and all the other vertices have at least two descendants. Assign labels $(\pp_0,\pp_1),(\pp_1,\pp_2),$ $\ldots,(\pp_{n-1},\pp_n)$ to the leaves as they are encountered while traversing the tree depth-first from left to right. Recursively, we assign the label $(\pp_{i_0},\pp_{i_1},\dots, \pp_{i_r})$ to any nonleaf vertex whose descendants have labels of the form $(\pp_{i_0},\ldots,\pp_{i_1}), (\pp_{i_1},\ldots, \pp_{i_2}),\ldots,$ $(\pp_{i_{r-1}},\ldots,\pp_{i_r})$. The labels of the nonleaf vertices then coincide with the faces of a partition in $\PPP(\pp_0,\pp_1,\ldots,\pp_n)$, and the labels of the leaves correspond to the outer edges unequal to $(\pp_0,\pp_n)$ of the full polygon. The following picture illustrates this construction with an example.

\noindent \includegraphics[scale=0.45]{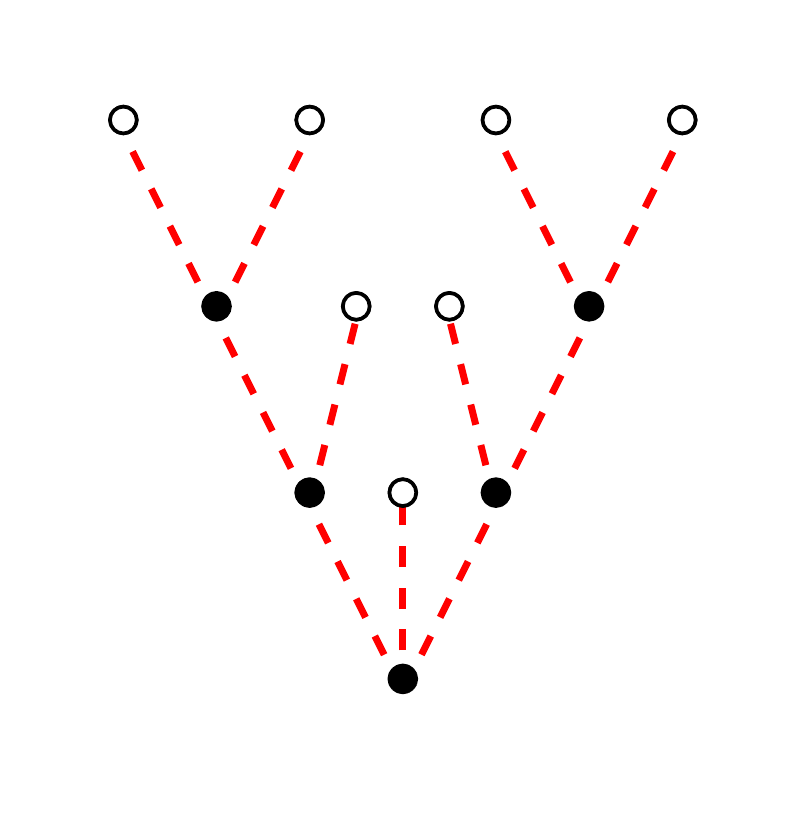}\quad \includegraphics[scale=0.45]{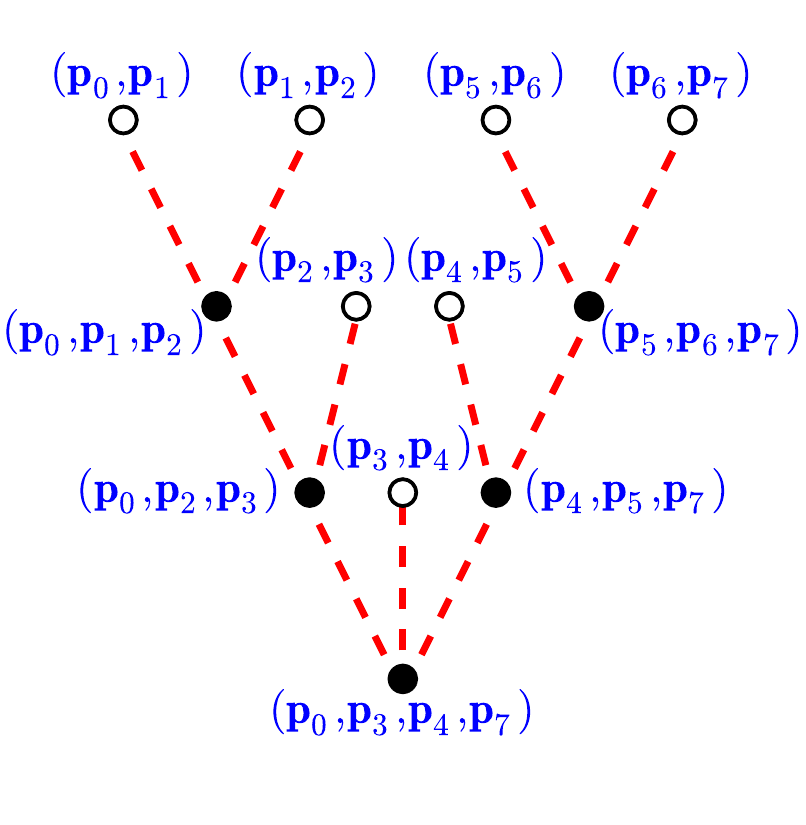}\quad \includegraphics[scale=0.45]{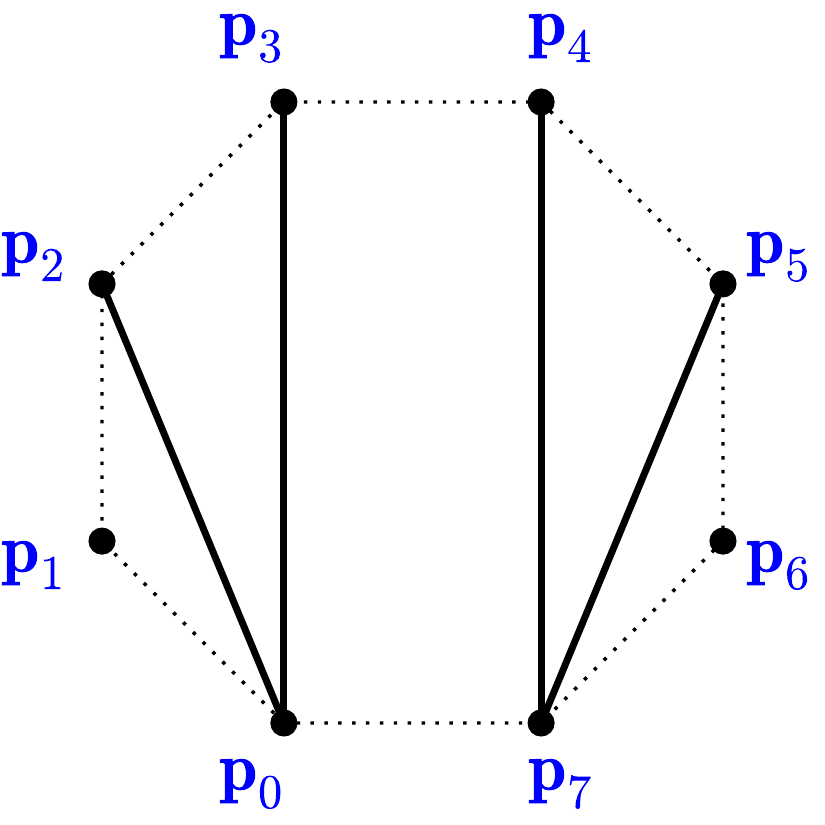}

Let $\TTT(\pp_0,\pp_1,\ldots,\pp_n)$ denote the set of plane trees with $n$ leaves for which all nonleaf vertices have at least two descendants. Represent each tree $\tau \in \TTT(\pp_0,\pp_1,\ldots,\pp_n)$ by its set $V(\tau)$ of \emph{nonleaf} vertices that are labeled as above, and label the leaves correspondingly. Note that $V(\tau) = F(\pi)$ whenever a tree $\tau$ and polygon partition $\pi$ are related via the above bijection. From this it follows that Equation \ref{eq:MainTheorem} can equivalently be stated in terms of plane trees as
\begin{equation}\tag{\ref{eq:MainTheorem}$^\prime$}\label{eq:MainTheorem2}
[\xx: \00,\nn ]y =
\sum_{\00 = \pp_0 < \pp_1 < \cdots < \pp_{|\nn|} = \nn}\ 
\sum_{\tau \in \TTT(\pp_0,\pp_1,\ldots,\pp_{|\nn|})}\ 
\prod_{(\ii_0,\ii_1,\ldots,\ii_k)\in V(\tau)} \{\ii_0 \ii_1\cdots \ii_k\}g.
\end{equation}

We wish to bring this equation into a form where we can distinguish the individual terms in the divided differences of $g$. For this, we replace $\TTT(\pp_0,\ldots,\pp_{|\nn|})$ by a structure $\TTT'(\pp_0,\ldots,\pp_{|\nn|})$ that encompasses all combinations of all terms in the expressions $\{\ii_0\cdots \ii_k\}g$. More precisely, every tree $\tau$ in $\TTT(\pp_0,\ldots,\pp_{|\nn|})$ will be replaced by several trees, one for each term in Equation \ref{eq:CurlyBrackets}, each of which is an extension of $\tau$ with an additional nonleaf vertex for every factor in the second line of Equation \ref{eq:CurlyBrackets}.

Let $\tau$ be one of the trees in $\TTT(\pp_0,\ldots,\pp_{|\nn|})$ corresponding to a polygon partition $\pi$. Any nonleaf vertex $\vv = (\ii_0, \ldots, \ii_k)$ in $V(\tau)$ defines, together with its direct descendants, a subtree $\star_\vv$ of $\tau$ called a \emph{[plane] star with root $\vv$}, in which we remember which descendants were (non)leaves. Note that the bijection of Lemma \ref{lem:DissectionsPlanarTrees} induces a bijection between $F(\pi)$ and the set $\Stars(\tau)$ of stars of nonleaf vertices of $\tau$. A star $\star$ is said to be \emph{of type} $(s_1^\star,\ldots, s_q^\star, t^\star) = (\ss^\star,t^\star)$, if the sequence of descendants of its root starts with
\begin{itemize}
\item[] $s_1^\star$ leaves with labels $(\aaa,\bb)$ satisfying $\bb - \aaa = \ee_1$, followed by
\item[] $s_2^\star$ leaves with labels $(\aaa,\bb)$ satisfying $\bb - \aaa = \ee_2$, followed by
\item[] $\ \vdots$ 
\item[] $s_q^\star$ leaves with labels $(\aaa,\bb)$ satisfying $\bb - \aaa = \ee_q$, followed by 
\item[] $t^\star$ nonleaves.
\end{itemize}
See Figure \ref{fig:Star0} for an example. Note that such a type does not exist for every star, as leaves can appear after nonleaves.

For every integer partition $k = s_1 + \cdots + s_q + t$ compatible with $\vv = (\ii_0,\ldots,\ii_k)$, we can extend $\star_\vv$ to a tree $\tau_{\vv}^{\ss,t}$ by inserting an edge at the leaves among the final $t$ descendants $(\ii_{|\ss|}, \ii_{|\ss| + 1}), (\ii_{|\ss| + 1}, \ii_{|\ss| + 2}), \ldots, (\ii_{|\ss|+t-1}, \ii_{|\ss| + t})$ of $\vv$. That is, we insert an edge for every factor in the second line of Equation \ref{eq:CurlyBrackets}. Note that if there are no such factors, then $\tau_\vv^{\ss,t} = \star_\vv$. Every star $\star$ in $\tau_\vv^{\ss,t}$ is then of some (necessarily unique) type $(\ss^\star,t^\star)$. Using these notions, one can write
\[ \{\ii_0\cdots\ii_k\}g = 
\sum_{\substack{\text{compatible}\\(\ss,t)}}\ 
\prod_{\star\, \in\, \Stars(\tau_{\vv}^{\ss,t})}\]
\[
\left(-\frac
{\big[i_0^{1\star}\cdots (i_0^{1\star} + s_1^\star) ;\ \cdots\ ; i_0^{q\star}\cdots (i_0^{q\star} + s_q^\star)\big| \ii_{|\ss^\star|}^\star\cdots \ii_{|\ss^\star|+t^\star}^\star\big]g}
{\big[i_0^{1\star};\ \cdots; i_0^{q\star}\big|\ii_0^\star \, \ii_{k^\star}^\star\big]g}
\right), \]
where each star $\star$ is of type $(\ss^\star,t^\star) = \big(s_1^\star,\ldots,s_q^\star,t^\star\big)$ and has root $\big(\ii_0^\star,\ldots,\ii_{k^\star}^\star\big)$, with 
\[ \big(i_0^{1\star},\ldots,i_0^{q\star}\big) = \ii_0^\star < \big(i_1^{1\star},\ldots,i_1^{q\star}\big) = \ii_1^\star < \cdots < \big(i_{k^\star}^{1\star},\ldots,i_{k^\star}^{q\star}\big) = \ii_{k^\star}^\star. \]

\begin{figure}
\begin{center}
\subfloat[]{\includegraphics[scale=0.65]{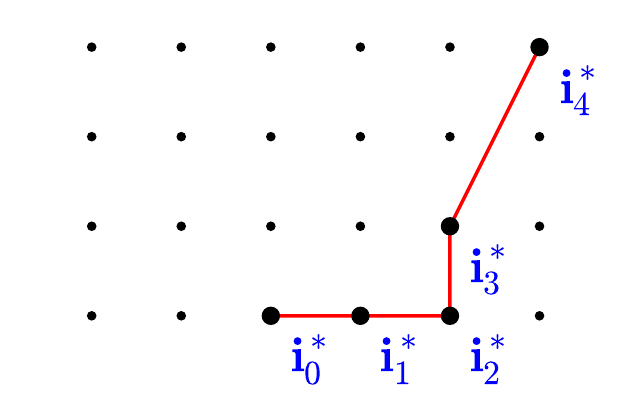}\label{fig:TypeA}}\qquad
\subfloat[]{\includegraphics[scale=0.65]{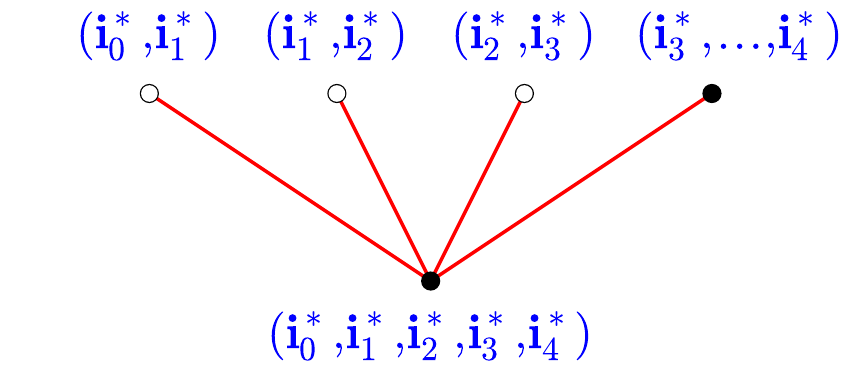}\label{fig:TypeB}}
\end{center}
\caption[A path and a corresponding star]{The figure to the left shows the path $(\ii_0^\star,\ii_1^\star,\ii_2^\star,\ii_3^\star,\ii_4^\star) = (2\ee_1, 3\ee_1, 4\ee_1, 4\ee_1 + \ee_2, 5\ee_1+3\ee_2)$.
The figure to the right shows a star $\star$ of type $(s_1^\star,s_2^\star,t^\star) = (2,1,1)$ with root $(\ii_0^\star,\ii_1^\star,\ii_2^\star,\ii_3^\star,\ii_4^\star)$, with the circles representing leaves and the discs nonleaves.}
\label{fig:Star0}
\end{figure}

Let $\TTT'(\pp_0,\ldots,\pp_{|\nn|})$ be the set of plane trees obtained by taking a plane tree in $\TTT(\pp_0,\ldots,\pp_{|\nn|})$ and replacing each of its stars $\star_\vv$ by $\tau_\vv^{\ss,t}$ for some $(\ss,t)$ compatible with $\vv$.
Equivalently, $\TTT'(\pp_0,\ldots,\pp_{|\nn|})$ is the set of plane trees with leaves $(\pp_0,\pp_1),(\pp_1,\pp_2),$ $\ldots,(\pp_{|\nn|-1}, \pp_{|\nn|})$ and nonleaves labeled accordingly, for which each star $\star$ is of some type $(\ss^\star,t^\star)\neq (\00,1)$.
Equation \ref{eq:MainTheorem2} can then be stated as
\begin{equation}\tag{\ref{eq:MainTheorem}$^{\prime\prime}$}\label{eq:MainTheorem3}
[\xx: \00,\nn ]y =
\sum_{\00 = \pp_0 < \pp_1 < \cdots < \pp_{|\nn|} = \nn}\ 
\sum_{\tau' \in \TTT'(\pp_0,\pp_1,\ldots,\pp_{|\nn|})}\ 
\prod_{\star\,\in\, \Stars(\tau')}
\end{equation}
\[ 
\left(-\frac
{[i_0^{1\star}\cdots (i_0^{1\star} + s_1^\star) ;\ \cdots\ ; i_0^{q\star}\cdots (i_0^{q\star} + s_q^\star)| \ii_{|\ss^\star|}^\star\cdots \ii_{|\ss^\star|+t^\star}^\star]g}
{[i_0^{1\star};\ \cdots; i_0^{q\star}|\ii_0^\star \, \ii_{k^\star}^\star]g}
\right), \]
where again each star $\star$ is of type $(\ss^\star,t^\star) = (s_1^\star,\ldots,s_q^\star,t^\star)$ and has root $(\ii_0^\star,\ldots,\ii_{k^\star}^\star)$, with 
\[ \big(i_0^{1\star},\ldots,i_0^{q\star}\big) = \ii_0^\star < \big(i_1^{1\star},\ldots,i_1^{q\star}\big) = \ii_1^\star < \cdots < \big(i_{k^\star}^{1\star},\ldots,i_{k^\star}^{q\star}\big) = \ii_{k^\star}^\star. \]

\begin{example}
To the path $(\pp_0,\pp_1,\pp_2,\pp_3) = (\00,\ee_1,\ee_1 + \ee_2, 2\ee_1 + \ee_2)$ correspond three trees in $\TTT(\pp_0,\pp_1,\pp_2,\pp_3)$,

\begin{center}
\begin{tabular}{ccc}
$\tau_1$ & $\tau_2$ & $\tau_3$\\
\includegraphics[scale=0.75]{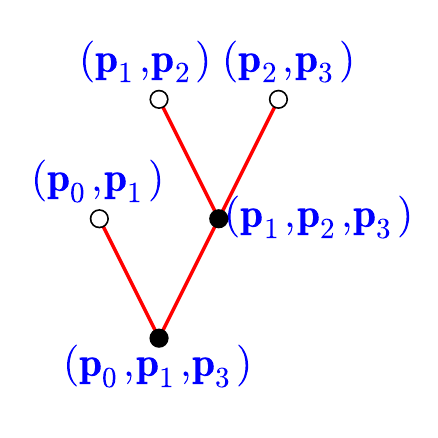} & 
\includegraphics[scale=0.75]{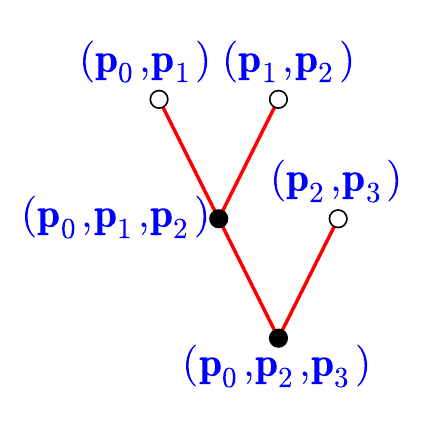} &
\includegraphics[scale=0.75]{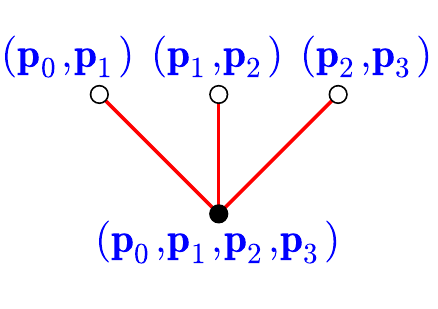}
\end{tabular}
\end{center}

Let us consider the first tree $\tau_1$. There are two tuples $(\ss_1,t_1) = (0,\ldots,0,2)$, $(\ss_2,t_2) = (1,0,\ldots,0,1)$ compatible with the nonleaf vertex $\vv = (\pp_0,\pp_1,\pp_3)$, and we can extend $\star_\vv$ to two corresponding trees

\medskip
\begin{center}
\begin{tabular}{cc}
$\tau_\vv^{\ss_1,t_1}$                & \qquad\qquad $\tau_\vv^{\ss_2,t_2} = \star_\vv$\\
\includegraphics[scale=0.75]{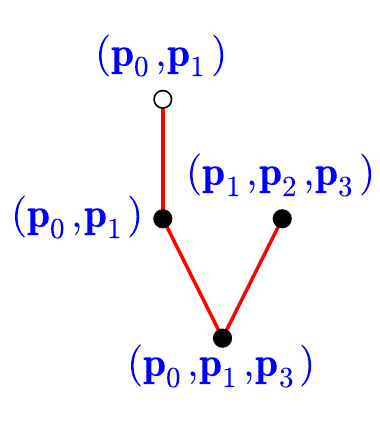} & \qquad\qquad \includegraphics[scale=0.75]{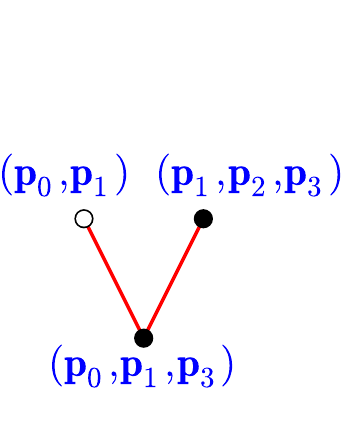}
\end{tabular}
\end{center}

\noindent Similarly, there are two tuples $(\ss_1',t_1') = (0,\ldots,0,2)$, $(\ss_2',t_2') = (0,1,0,\ldots,0,1)$ compatible with the other nonleaf vertex $\vv' = (\pp_1,\pp_2,\pp_3)$ of $\tau_1$, and we find two trees 

\begin{center}
\begin{tabular}{cc}
$\tau_{\vv'}^{\ss_1',t_1'}$ & $\tau_{\vv'}^{\ss_2',t_2'}$ \\
\includegraphics[scale=0.75]{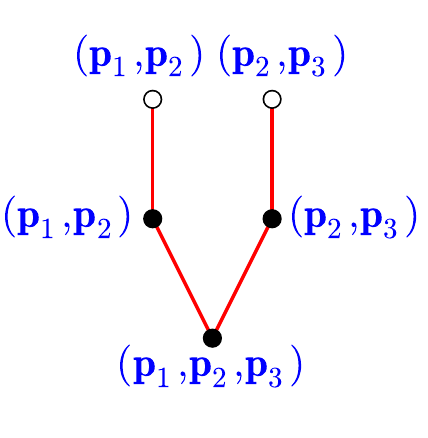} & \qquad\quad \includegraphics[scale=0.75]{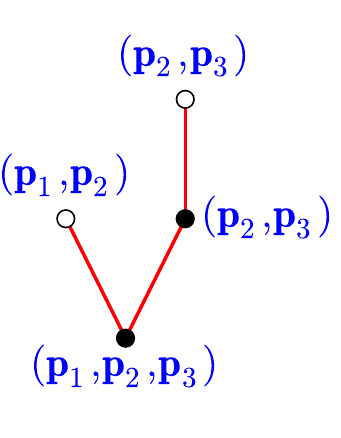}
\end{tabular}
\end{center}

\noindent corresponding to $\star_{\vv'}$. It follows that the tree $\tau_1$ yields $2\times 2 = 4$ different trees in $\TTT'(\pp_0,\pp_1,\pp_2,\pp_3)$. Analogously, one can check that $\tau_2$ yields $1\times 3$ trees and $\tau_3$ yields $3$ trees.
\end{example}

We end this section with a Lemma that appears as Theorem 5.3.10 in \cite{Stanley99} and is needed to compute the coefficients in Equation \ref{eq:Wilde}.
\begin{lemma}\label{lem:NumberPlanarTrees} For any $k\in \NN$, there are
\[ \frac{1}{r_0 + r_1 + \cdots + r_k }{r_0 + r_1 + \cdots + r_k\choose r_0,r_1,\ldots, r_k} \]
different plane trees with $r_0$ vertices with $0$ descendants (leaves), $r_1$ vertices with $1$ descendant, $\ldots$, $r_k$ vertices with $k$ descendants, and no vertices with more than $k$ descendants.
\end{lemma}

\section{Implicit higher partial derivatives}\label{sec:ImplicitDerivatives}
Whenever $g$ and $y$ are sufficiently smooth, coalescing the grid points in Equation \ref{eq:MainTheorem3} results in a formula for the derivatives of $y$ in terms of the derivatives of $g$. We show that this formula generalizes a formula that appears as Equation 7 in \cite{Wilde08}, which corrects a misprint in an earlier formula by Comtet and Fiolet \cite{ComtetFiolet74}.

The formula, as stated in Equation \ref{eq:Wilde}, uses some notation for $(q+1)$-dimensional partitions. If $(\nn,m)\in \NN^q\times \NN$ is a nonzero tuple of nonnegative integers, then a \emph{$(q+1)$-dimensional partition} $p$ of $(\nn,m)$, denoted by $p\vdash (\nn,m)$, is a multiset $\big\{(\ss_1,t_1), \ldots,$ $ (\ss_r,t_r)\big\}$
of nonzero tuples in $\NN^q\times \NN$ that sum to $(\nn, m)$ when counting multiplicities. We let $|p| = r$ denote the number of terms in the partition $p$, counting the multiplicity $\mu_{p;\ss,t}$ of each tuple $(\ss,t)$ in $p$. 

Let $y$ be implicitly defined by $g$ as in Equation \ref{eq:ImplicitlyDefined}. We introduce the shorthands
\[
y_\nn = y_\nn (\xx) := \frac{\partial^{|\nn|} y}{\partial \xx^\nn}(\xx),\qquad 
g_{\ss,t} = g_{\ss,t} \big(\xx,y(\xx)\big) := \frac{\partial^{|\ss| + t} g}{\partial \xx^\ss y^t}\big(\xx, y(\xx)\big).
\]
As the multiplicities $\mu_{p;\ss,t}$ sum to $|p|$, the multinomial coefficient
\[ {|p|\choose \ldots,\mu_{p;\ss,t},\ldots} :=
\frac{|p|!}{\mu_{p;\ss_1,t_1}!\cdots \mu_{p;\ss_l,t_l}!}
\]
is well defined for any partition
\[
p = \{\underbrace{(\ss_1,t_1),\ldots,(\ss_1,t_1)}_{\mu_{p;\ss_1,t_1}},\ \cdots, \underbrace{(\ss_l,t_l),\ldots,(\ss_l,t_l)}_{\mu_{p;\ss_l,t_l}}\}.
\]

\begin{corollary}\label{cor:Wilde} If $y$ and $g$ are sufficiently smooth and related by Equation \ref{eq:ImplicitlyDefined}, then, for any $\xx\in U$ and nonzero $\nn\in \NN^q$,
\begin{equation}\label{eq:Wilde} \frac{y_\nn(\xx) }{\nn!} = 
\sum_{\substack{p\, \vdash\, (\nn, |p| - 1)\\ (\00,1)\notin p}}
\frac{1}{|p|} {|p| \choose \ldots,\mu_{p;\ss,t},\ldots }
\prod_{(\ss,t)\in p} \left( - \frac{1}{\ss ! t!} \frac{g_{\ss,t}\big(\xx,y(\xx)\big)}{g_{\00,1}\big(\xx,y(\xx)\big)} \right).
\end{equation}
\end{corollary}

Here the product is understood to be of $\mu_{p;\ss,t}$ copies for every distinct element $(\ss,t)$ of the multiset $p$. The self-referring nature of the summation makes it not directly obvious that there is only a finite number of partitions $p$ of this form for any $\nn$. Given such a partition $p$, only $a\leq |\nn|$ of its elements $(\ss,t)$ satisfy $\ss\neq \00$. Since $p$ does not have $(\00, 1)$ as an element, each of the final coordinates of the $b$ remaining elements of $p$ is at least two. Then $2b \leq |p| - 1 = a + b - 1$ implies that $p$ contains at most $|p| = a + b \leq 2a - 1\leq 2|\nn| - 1$ elements. This bound guarantees that any partition $p$ should sum to a tuple smaller than $(\nn, 2|\nn| - 1)$, implying that, for given $\nn$, there is but a finite number of multisets $p$ of nonzero tuples in $\NN^q \times \NN$ satisfying $p \vdash (\nn, |p| - 1)$ and $(\00, 1)\notin p$.

On the other hand, the existence of partitions of this form can be seen by taking simple examples. For example, for $q=2$ and $\nn = (1,0), (2,0), (1,1)$ one finds partitions

\begin{align}
\nn = (1,0): \ & \{(1,0,0)\} \vdash (1,0,0),\notag \\
\nn = (2,0): \ & \{(2,0,0)\} \vdash (2,0,0),\notag \\
                    & \{(1,0,1),(1,0,0)\} \vdash (2,0,1),\notag \\
                    & \{(0,0,2),(1,0,0),(1,0,0)\} \vdash (2,0,2),\notag \\
\nn = (1,1): \  & \{(1,1,0)\} \vdash (1,1,0),\label{eq:Partitions11} \\
                    & \{(1,0,0),(0,1,1)\},\  \{(0,1,0),(1,0,1)\} \vdash (1,1,1), \notag\\
                    & \{(1,0,0), (0,1,0), (0,0,2)\} \vdash (1,1,2). \notag
\end{align}
For these $\nn$, the Corollary claims
\begin{align}
  y_{1,0} = & - \frac{g_{1,0,0}}{g_{0,0,1}},\notag \\
  y_{2,0} = & - \frac{g_{2,0,0}}{g_{0,0,1}} + 2\frac{g_{1,0,1} g_{1,0,0}}{g_{0,0,1}^2} - \frac{g_{0,0,2} g_{1,0,0}^2}{g_{0,0,1}^3},\notag \\
  y_{1,1} = & - \frac{g_{1,1,0}}{g_{0,0,1}} + \frac{g_{1,0,0} g_{0,1,1}}{g_{0,0,1}^2} + \frac{g_{0,1,0} g_{1,0,1}}{g_{0,0,1}^2} - \frac{g_{1,0,0}g_{0,1,0}g_{0,0,2}}{g_{0,0,1}^3}. \label{eq:Wilde11}
\end{align}

Coalescing the grid to a single point $\xx_0$ in Equation \ref{eq:MainTheorem3}, one finds that
\begin{equation}\label{eq:IntermediateWilde}
\frac{y_\nn(\xx_0)}{\nn!} =
\sum_{\00 = \pp_0 < \pp_1 < \cdots < \pp_{|\nn|} = \nn}\ 
\sum_{\tau' \in \TTT'(\pp_0,\pp_1,\ldots,\pp_{|\nn|})}
\end{equation}
\[ \prod_{\star\,\in\, \Stars(\tau')} \left(-\frac{1}{\ss^\star!t^\star!}\frac{g_{\ss^\star,t^\star}\big(\xx_0,y(\xx_0)\big)}{\ \ g_{\00,1}\big(\xx_0,y(\xx_0)\big)} \right),\]
where each star $\star$ is of type $(\ss^\star, t^\star)$. Clearly Equations \ref{eq:Wilde} and \ref{eq:IntermediateWilde} are in a similar form. The only difference seems to be that in Equation \ref{eq:Wilde} equal terms are grouped together into one term with a coefficient. It is not surprising that there are duplicate terms in Equation \ref{eq:IntermediateWilde}, as each term depends only on the types of the stars, not on how these stars are connected to form a tree.

We first provide an example that introduces the flavor of the proof of Corollary \ref{cor:Wilde}.
\begin{example}
Let $\nn = (1, 1)$. Coalescing the grid to a single point $(x^1_0, x^2_0)$ in Equation \ref{eq:DivDiff11}, one finds
\begin{align}\label{eq:IntermediateWilde11}
y_{1,1} =\,& - \frac{1}{2} \frac{g_{0,0,2} g_{1,0,0} g_{0,1,0}}{g_{0,0,1}^3} + \frac{g_{1,0,1} g_{0,1,0}}{g_{0,0,1}^2} - \frac{g_{1,1,0}}{g_{0,0,1}} \\
         \,& - \frac{1}{2} \frac{g_{0,0,2} g_{0,1,0} g_{1,0,0}}{g_{0,0,1}^3} + \frac{g_{0,1,1} g_{1,0,0}}{g_{0,0,1}^2} \notag
\end{align}
at this point $(x^1_0, x^2_0)$. Clearly Equations \ref{eq:Wilde11} and \ref{eq:IntermediateWilde11} are equivalent.
In this example we hint at how the terms in these equations are related, suggesting a link that generalizes to the generic construction in the proof of Corollary \ref{cor:Wilde}.

First of all note that in both Equations \ref{eq:Wilde11} and \ref{eq:IntermediateWilde11} the denominators can be determined from their numerators. Taking for granted that the coefficients agree, it therefore suffices to check that, for these equations, the monomials of the numerators of their terms agree.

For each monomial in Equation \ref{eq:IntermediateWilde11}, the orders of the derivatives in the numerators form a multiset $p$ of triples in $\NN^2\times \NN$ with $(0,0,1)\notin p$ and $p\vdash (1,1,|p|-1)$. It follows that every monomial in Equation \ref{eq:IntermediateWilde11} appears in Equation \ref{eq:Wilde11} as well.

\begin{table}
\begin{center}
\begin{tabular}{c|c|c|c|c|c}
\hline
$p$     & \multicolumn{2}{|c|}{\small $\{(0,0,2),(1,0,0),(0,1,0)\}$}         & {\small $\{(0,1,1),(1,0,0)\}$}               & 
          {\small $\{(1,0,1),(0,1,0)\}$}                     & {\small $\{(1,1,0)\}$}                                       \\\hline
$\star$     & \multicolumn{2}{|c|}{\includegraphics[scale=0.55]{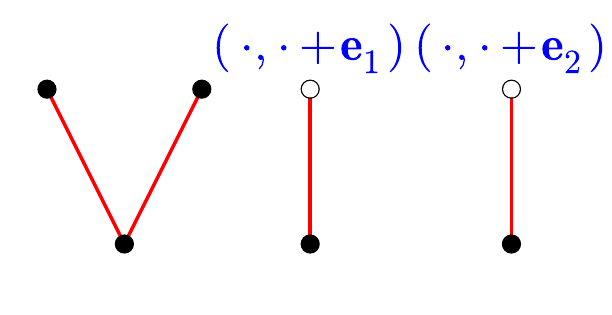}} & \includegraphics[scale=0.55]{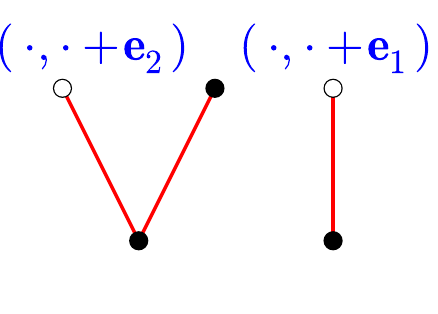} & 
          \includegraphics[scale=0.55]{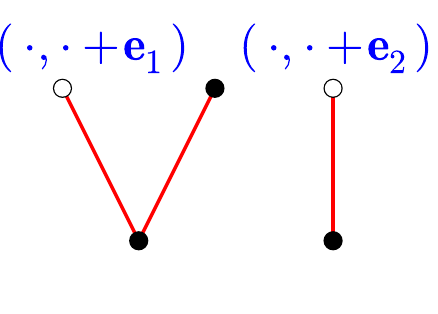}       & \includegraphics[scale=0.55]{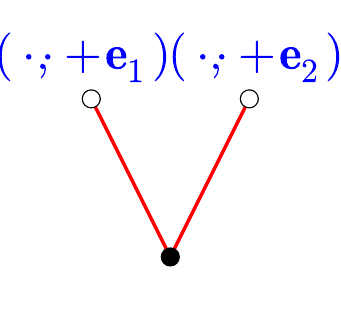}                 \\\hline
$\tau'$ & $\tau'_1$ & $\tau'_2$ & $\tau'_3$ & $\tau'_4$ & $\tau'_5$ \\
        & \includegraphics[scale=0.55]{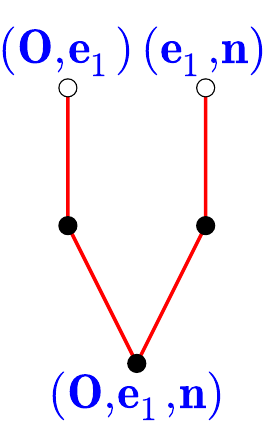}  &
          \includegraphics[scale=0.55]{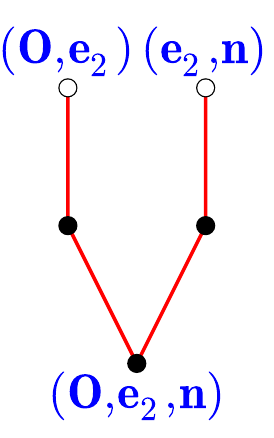} & \includegraphics[scale=0.55]{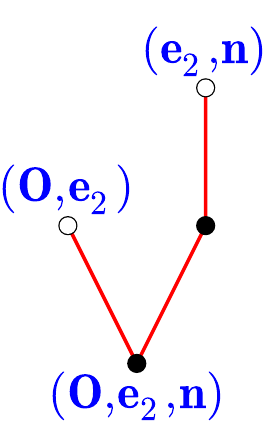}              &
          \includegraphics[scale=0.55]{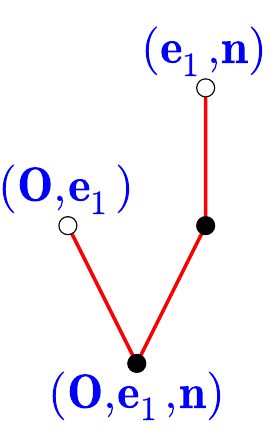}    & \includegraphics[scale=0.55]{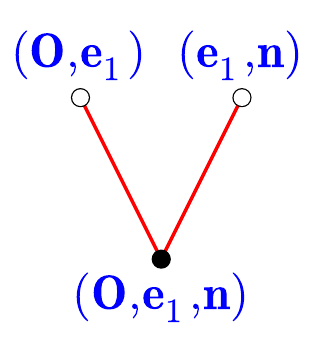}  \\
\hline
\end{tabular}
\end{center}
\caption[Multisets, stars, and plane trees]{For $\nn = (1,1)$, the first row lists the multisets $p$ with $(0,0,1)\notin p$ and $p\vdash (\nn,|p| - 1)$. The second row depicts the stars $\star$ associated to each of these multisets. The third row shows the different trees $\tau'$ that can be formed by connecting these stars, together with the labels of their vertices.}
\label{tab:AssemblingTrees11}
\end{table}

Conversely, we show that every monomial in Equation \ref{eq:Wilde11} appears in Equation \ref{eq:IntermediateWilde11} as well, by pointing out which paths $(0,0) = \pp_0 < \pp_1 < \pp_2 = (1,1)$ and trees $\tau'\in \TTT'(\pp_0,\pp_1,\pp_2)$ correspond to it. For a given monomial in Equation \ref{eq:Wilde11}, let $p$ be the corresponding partition in Equation \ref{eq:Partitions11}. To each triple $(s_1,s_2,t)$ in $p$ we associate a star of type $(s_1,s_2,t)$. As can be seen in Table \ref{tab:AssemblingTrees11}, one can, for each multiset $p$, connect these stars to a plane tree $\tau'$ with two leaves, and sometimes there are several ways to do this.
After demanding the first coordinate of the label of the left leaf to be $(0,0)$, there is only one way to label the leaves of $\tau'$ that agrees with the types of stars of $\tau'$. Thus we find three trees $\tau'_1,\tau'_4,\tau'_5\in \TTT'\big((0,0),(1,0),(1,1)\big)$ and two trees $\tau'_2,\tau'_3\in \TTT'\big((0,0),(0,1),(1,1)\big)$, each of which corresponds to a term in Equation \ref{eq:IntermediateWilde11}.
\end{example}

\begin{proof}[of Corollary \ref{cor:Wilde}]
In both Equation \ref{eq:Wilde} and \ref{eq:IntermediateWilde}, each term comprises some coefficient and a monomial in the symbols $g_{\ss,t}$ divided by a power of $g_{\00,1}$ of the same total degree. As this monomial uniquely determines the denominator, the following three steps suffice to show that Equations \ref{eq:Wilde} and \ref{eq:IntermediateWilde} are equivalent.

\emph{1. Every term in Equation \ref{eq:IntermediateWilde} appears also in Equation \ref{eq:Wilde}.}
Consider an arbitrary term $T$ in the right hand side of Equation \ref{eq:IntermediateWilde}. This term arises from picking a path $\00 = \pp_0 < \cdots < \pp_{|\nn|} = \nn$ and a tree $\tau' \in \TTT'(\pp_0,\ldots,\pp_{|\nn|})$. Let $p = \big\{(\ss_1,t_1),\ldots,(\ss_{|p|}, t_{|p|})\big\}$ be the multiset of types of the stars in $\tau'$. That is, $p$ is the multiset of orders of the derivatives in the numerator of $T$.

Since for any tree in $\TTT'(\pp_0,\ldots,\pp_{|\nn|})$ the steps made by its leaves sum to $\nn$, it follows that $\ss_1 + \cdots + \ss_{|p|} = \nn$. Moreover, with the exception of the root of $\tau'$, the root of each star in $\Stars(\tau')$ connects to one of the $t_1 + \cdots + t_{|p|}$ nonleaf descendants of the stars, implying that $t_1 + \cdots + t_{|p|} = |p| - 1$. As $\tau'$ has no vertices with precisely one nonleaf descendant, none of the types in $p$ can be equal to $(\00,1)$. The orders of the derivatives in the numerator of Equation \ref{eq:IntermediateWilde} therefore constitute a multiset $p$ with elements in $\NN^{q}\times \NN$ for which $(\00,1)\notin p$ and $p \vdash (\nn,|p| - 1)$. We conclude that, up to coefficients, each term in Equation \ref{eq:IntermediateWilde} appears as a term in Equation \ref{eq:Wilde} as well.

\emph{2. Every term in Equation \ref{eq:Wilde} appears also in Equation \ref{eq:IntermediateWilde}.} Suppose we are given a multiset $p = \{(\ss_1,t),\ldots,(\ss_{|p|}, t_{|p|})\}$ of tuples in $\NN^q\times \NN$ satisfying $(\00,1)\notin p$ and $p \vdash (\nn, |p| - 1)$ as in Equation \ref{eq:Wilde}. 
To each element $(\ss,t)$ of $p$, we associate a star of type $(\ss,t)$ whose labels are yet to be determined. Because the number of stars $|p|$ is one more than the sum of the nonleaf descendants $t_1 + \cdots + t_{|p|}$, one can always connect these stars to a plane tree $\tau'$ with $|\ss_1| + \cdots + |\ss_{|p|}| = |\nn|$ leaves.
Moreover, since each star in $\Stars(\tau')$ is of some type unequal to $(\00,1)$, there is a unique path $\00 = \pp_0 < \cdots < \pp_{|\nn|} = \nn$ and corresponding labeling of the leaves of $\tau'$ such that $\tau' \in \TTT'(\pp_0,\ldots,\pp_{|\nn|})$.
We conclude that, up to coefficients, each term in Equation \ref{eq:Wilde} appears as a term in Equation \ref{eq:IntermediateWilde} as well.

\emph{3. Corresponding terms have equal coefficients}.
Now we have shown that the monomials in Equation \ref{eq:Wilde} are the same as those in Equation \ref{eq:IntermediateWilde}, it remains to show that their coefficients agree. Every term in Equation \ref{eq:IntermediateWilde} corresponding to a tree $\tau'\in \TTT'(\pp_0,\ldots,\pp_{|\nn|})$, with $p$ the multiset of the types of the stars of $\tau'$, will contribute $(-1)^{\# \Stars(\tau')} = (-1)^{|p|}$ to the term in Equation \ref{eq:Wilde} corresponding to $p$. The coefficients can be shown to agree, therefore, by counting, for every multiset $p$ with $(\00,1)\notin p \vdash(\nn,|p|-1)$, the number of different plane trees that can be formed by connecting the stars of types corresponding to the elements of $p$.

Let $p$ be as in Equation \ref{eq:Wilde}. Let us call two elements $(\ss_i, t_i),$ $(\ss_j, t_j)$ of $p$ \emph{equivalent} whenever $t_i = t_j$. The equivalence classes form a new multiset $p'$ in which each element $[(\ss_i,t)]$, or simply $t$ for short, has multiplicity $\mu_t := \sum_{\ss\geq \00} \mu_{p;\ss,t}$.
Clearly $p'$ has the same number of elements as $p$. Associate with each $t\in p'$ a star with $t$ descendants. Here we think of a leaf as a star with $0$ descendants. By Lemma \ref{lem:NumberPlanarTrees}, one can construct 
\[ \frac{1}{|p'|}{|p'| \choose \mu_0, \mu_1,\ldots} \]
different plane trees from these stars. For the vertices with $t$ descendants of any such plane tree, the multinomial coefficient of $\{ \mu_{p;\ss,t} : \ss\geq \00 \}$ gives the number of ways to reinsert the leaves. It follows that we can form
\[ \frac{1}{|p'|}{|p'|\choose \mu_0,\mu_1,\ldots} \prod_{t\geq 0} {\mu_t\choose \ldots, \mu_{p;\ss,t},\ldots} = \frac{1}{|p|}{|p|\choose \ldots,\mu_{p;\ss,t},\ldots} \]
different plane trees from the stars corresponding to $p$. This agrees with the coefficient in Equation \ref{eq:Wilde}. We conclude that Equations \ref{eq:Wilde} and \ref{eq:IntermediateWilde} are equivalent.
\end{proof}

\section*{Acknowledgments}
I am grateful to Paul Kettler, Claudine Mitschi, Ragni Piene, and the anonymous referee, for kindly providing many useful comments on a draft of this paper.

\section*{References}

\bibliographystyle{amsxport}

\end{document}